\documentclass{amsart}
\usepackage{enumerate}
\newtheorem*{acknowledgements}{Acknowledgements}
\def\oo{\mathscr{O}}
\def\Id{\mathrm{Id}}
\def\pp{\mathscr{P}}
\def\As{\mathrm{As}}
\def\bJ{\mathbf{J}}
\def\bl{\boldsymbol{\lambda}}
\def\bk{\boldsymbol{\kappa}}

\def\bt{\boldsymbol{\tau}}
\def\det{\mathrm{det}}
\def\Cusp{\mathrm{Cusp}}

\usepackage{amsmath,amssymb,mathrsfs}
\usepackage{xcolor}

\newcounter{makeconstant}
\newenvironment{makeconstant}%
{\refstepcounter{makeconstant}}%
{}

\theoremstyle{plain}
\newtheorem{theorem}[equation]{Theorem}
\newtheorem{lemma}[equation]{Lemma}
\newtheorem{corollary}[equation]{Corollary}
\newtheorem{proposition}[equation]{Proposition}

\theoremstyle{definition}

\newtheorem{definition}[equation]{Definition}

\newtheorem{remark}[equation]{Remark}

\newtheorem*{remark*}{Remark}

\numberwithin{equation}{section} 

\def\E{{\rm E}}
\def\F{{\rm F}}
\def\G{{\rm G}}
\def\H{{\rm H}}
\def\I{{\rm I}}
\def\J{{\rm J}}
\def\K{{\rm K}}
\def\L{{\rm L}}

\def\N{{\rm N}}

\def\P{{\rm P}}

\def\R{{\rm R}}

\def\T{{\rm T}}

\def\W{{\rm W}}
\def\X{{\rm X}}

\def\Z{{\rm Z}}

\def\Fl{\overline{\mathbb{F}_\ell}}
\def\Ql{\overline{\mathbb{Q}_\ell}}
\def\Zl{\overline{\mathbb{Z}_\ell}}

\def\Cc{\mathscr{C}}

\def\Ww{\mathcal{W}}

\def\kk{\mathrm{k}}
\def\ll{\mathrm{l}}

\def\so{\mathsf{o}}

\def\k{{\kappa}}

\def\o{\mathfrak{o}}

\def\w{\varpi}

\def\({\left(}
\def\){\right)}
\def\>{\geqslant}
\def\<{\leqslant}

\def\Hom{\operatorname{Hom}}

\def\GL{\operatorname{GL}}
\def\Gal{\operatorname{Gal}}

\def\Res{\operatorname{Res}}

\def\ind{\operatorname{ind}}

\def\dim{\operatorname{dim}}

\def\val{\operatorname{val}}

\def\St{\operatorname{St}}

\def\presuper#1#2%
  {\mathop{}%
   \mathopen{\vphantom{#2}}^{#1}%
   \kern-\scriptspace%
   #2}

\title{Characterisation of the poles of the $\ell$-modular Asai~$\L$-factor}
\author{Robert Kurinczuk}
\address{Robert Kurinczuk, Department of Mathematics, Imperial College, London, SW7 2AZ, United Kingdom.}
\email{robkurinczuk@gmail.com}
\author{Nadir Matringe}
\address{Universit\'e de Poitiers, Laboratoire de Math\'ematiques et Applications,
T\'el\'eport 2 - BP 30179, Boulevard Marie et Pierre Curie, 86962, Futuroscope Chasseneuil Cedex. France.}
\email{nadir.matringe@math.univ-poitiers.fr}
\subjclass[2010]{22E50; 11F70}

\begin{document}
\maketitle

\begin{abstract}
Let~$\F/\F_\so$ be a quadratic extension of non-archimedean local fields, set~$\G=\GL_n(F)$, 
$\G_\so=\GL_n(\F_\so)$ and let~$\ell$ be a prime number different from the residual characteristic of~$\F$. For a complex cuspidal representation~$\pi$ of~$\G$, the Asai~$\L$-factor~$\L_{\As}(\X,\pi)$ has a pole at~$\X=1$ if and only if~$\pi$ is~$\G_\so$-distinguished.   In this paper we solve the problem of characterising the occurrence of a pole at~$\X=1$ of~$\L_{\As}(\X,\pi)$ when~$\pi$ is an~$\ell$-modular cuspidal representation of~$\G$: we show that~$\L_{\As}(\X,\pi)$ has a pole at~$\X=1$ if and only if~$\pi$ is a \emph{relatively banal} distinguished representation; namely~$\pi$ is~$\G_\so$-distinguished but not~$\vert\det(~ )|_{\F_\so}$-distinguished.  This notion turns out to be an exact analogue for the symmetric space 
$\G/\G_\so$ of M\' inguez and S\'echerre's notion of banal cuspidal~$\Fl$-representation of~$\G_\so$.
  Along the way we compute the Asai~$\L$-factor of all cuspidal~$\ell$-modular representations of~$\G$ in terms of type theory, and prove new results concerning lifting and reduction modulo~$\ell$ of distinguished cuspidal representations. Finally, we determine when the natural~$\G_{\so}$-period on the Whittaker model of a distinguished cuspidal representation of $\G$ is nonzero. 
\end{abstract}

\section{Introduction}\label{section intro}

Let~$\F/\F_\so$ be a quadratic extension of non-archimedean local fields of residual characteristic~$p\neq 2$, and set~$\G=\GL_n(\F)$ and~$\G_\so=\GL_n(\F_\so)$.  An irreducible representation of~$\G$ is said to be \emph{distinguished by~$\G_\so$} if it possesses a nonzero~$\G_\so$-invariant linear form.  In the case of complex representations, the equality of the Asai~$\L$-factor defined by the Rankin--Selberg method and its Galois avatar (\cite{AR}, \cite{MatringeDistinguishedGeneric}) provides a bridge between functorial lifting from the quasi-split unitary group~$\mathrm{U}_n(\F/\F_\so)$ and~$\G_\so$-distinction of discrete series representations of~$\G$: a discrete series of~$\G$ is a (stable or unstable depending on the parity of~$n$) lift of a (necessarily discrete series) representation of 
$\mathrm{U}_n(\F/\F_\so)$ if and only if the Asai~$\L$-factor of its Galois parameter has a pole at~$\X=1$ (\cite{MokEndoscopic}, \cite{GGP1}), whereas it is~$\G_{\so}$-distinguished if and only if its Asai~$\L$-factor obtained by the Rankin-Selberg method has a pole at~$\X=1$ (\cite{Kable}, \cite{AKT}).\\

Recently, motivated by the study of congruences between automorphic representations, there has been great interest in studying representations of~$\G$ on vector spaces over fields of positive characteristic~$\ell$.  There are two very different cases, when~$\ell=p$ and when~$\ell\neq p$. This article focuses on the latter~$\ell\neq p$ case, where there is a theory of Haar measure which allows us to define Asai~$\L$-factors via the Rankin--Selberg method as in the complex case (Section \ref{section Asai}).\\
 
The aim of this article is to show that in this case a connection remains between the poles of Asai~$\L$-factor and distinction, however this connection no longer characterises distinction, but a more subtle notion which we call 
\textit{relatively banal} distinction. The easiest way to state that a cuspidal distinguished~$\ell$-modular representation is relatively banal is to say that it is not~$|\det(\ \cdot \ )|_{\F_\so}$-distinguished, where~$|\det(\ \cdot \ )|_{\F_\so}$ is considered as an~$\Fl$-valued character, but other compact definitions can also be given in terms of type theory as well as in terms of its supercuspidal lifts:

\begin{proposition}[Defintion \ref{definition relbanaldef}, Theorem \ref{thm relatively banal distinction lift} and 
Corollary \ref{cor relatively banal simple definition}]
Let~$\pi$ be an~$\ell$-modular cuspidal distinguished representation of~$\G$. Then the following are equivalent, and when they are satisfied we say that $\pi$ is \emph{relatively banal}:
\begin{enumerate}
\item~$\pi$ is not~$|\det(\ \cdot \ )|_{\F_\so}$-distinguished.
\item All supercuspidal lifts of~$\pi$ are distinguished by an unramified character of~$\G_\so$.
\item~$q_\so^{n/e_\so(\pi)}\not\equiv 1[\ell]$, where~$e_\so(\pi)$ denotes the invariant associated to~$\pi$ in 
\cite[Lemma 5.10]{AKMSS} (see Section \ref{GSDTsection}).
\end{enumerate}
\end{proposition}
Relatively banal for~$\G_\so$-distinguished cuspidal representations turns out to be the exact analogue of the definition of \emph{banal cuspidal representations} of $\G_\so$ (see \cite[Remark 8.15]{MSDuke} and \cite{MSbanal}) after one identifies the cuspidal (irreducible) representations of~$\G_\so$ with the $\Delta(\G_\so)$-distinguished cuspidal (irreducible) representations of~$\G_\so\times\G_\so$, where~$\Delta:\G_\so\rightarrow\G_\so\times\G_\so$ is the diagonal embedding, as we explain in Section \ref{section comparison banal rel banal}.\\

The main theorem of this paper characterises the poles of the Asai~$\L$-factor:

\begin{theorem}[Theorem \ref{thm Pole of Asai}]\label{Main}
Let~$\pi$ be a cuspidal~$\ell$-modular representation, then~$\L_{\As}(\X,\pi)$ has a pole at~$\X=1$ if and only if~$\pi$ is relatively banal distinguished.
\end{theorem}

Note that the proof of the above theorem is completely different from the proof of characterisation theorem in the complex case (see Remark \ref{remark Kable} for more on the comparison of the proofs). Here, 
we show   
that the Asai~$\L$-factor of a 
cuspidal~$\ell$-modular representation is equal to~$1$ whenever~$\pi$ is not the unramified twist of a relatively banal representation using Theorem \ref{thm relatively banal distinction lift}, which is the characterisation of relatively banal in terms of supercuspidal lifts.
Then when~$\pi$ is the unramified twist of relatively banal representation, following our paper \cite{KMNagoya}, we get an explicit formula for~$\L_{\As}(\X,\pi)$ 
in Theorem \ref{thm AsaiLfunctioncomputation} from the test vector computation of \cite{AKMSS}, which thanks to the relatively banal 
assumption (more precisely its type theory version) reduces modulo~$\ell$ without vanishing. We then deduce Theorem \ref{Main} 
from this computation, together with the computation of the group of unramified characters $\mu$ of $\G_\so$ such that~$\pi$ is 
$\mu$-distinguished (Corollary \ref{cor cardinality of Xso}). \\

Finally, denoting by $\N$ the unipotent radical of the group of upper triangular matrices in $\G$, ~$\Z_\so$ the centre of~$\G_\so$ and~$\N_{\so}$ the group $\N\cap \G_\so$, the most natural~$\G_\so$-invariant linear form to consider on the Whittaker model of an~$\ell$-modular cuspidal representation~$\pi$ with respect to a distinguished non degenerate character of~$\N$ trivial on $\N_\so$ is the local period 
\[\mathcal{L}_\pi:\W\mapsto \int_{\Z_\so\N_\so\backslash \G_\so} \W(h)dh.\] 
In fact this period plays an essential role in the proof of Theorem \ref{Main} over the field of complex numbers (see Remark \ref{remark Kable}). One of the main differences in the~$\ell$-modular setting is that $\mathcal{L}_\pi$ can be zero even when $\pi$ is distinguished. Not only do we show that it can vanish but we say exactly when it does: 

\begin{theorem}[Theorem \ref{thm nonvanishingperiod}]
Let~$\pi$ be a cuspidal distinguished $\ell$-modular representation of~$\GL_n(\F)$. Then the local period~$\mathcal{L}_\pi$ is nonzero if and only if the following two properties are satisfied:
\begin{enumerate}
\item~$\pi$ is relatively banal.
\item~$\ell$ does not divide $e_\so(\pi)$, in other words if $\widetilde{\pi}$ is a lift of $\pi$, the $\ell$-adic valuation of $n$ is the same as the $\ell$-adic valuation 
of the number of $\ell$-adic unramified characters $\widetilde{\mu}_\so$ of $\G_\so$ such $\widetilde{\pi}$ is $\widetilde{\mu}_\so$-distinguished.
\end{enumerate}
\end{theorem}

This theorem is related to the vanishing modulo~$\ell$ of a rather interesting and subtle scalar related, after fixing an isomorphism~$\mathbb{C}\simeq \Ql$, to a quotient of the formal degree of a complex cuspidal representation of a unitary group by the formal degree of its base change to~$\GL_n(\F)$, see Remark \ref{remarklinearforms} for a precise statement.

In light of Theorem \ref{Main}, the role of the Asai~$\L$-factor in the study of distinguished 
representations will be less important in the~$\ell$-modular setting as some~$\ell$-modular distinguished representations have Asai~$\L$-factors equal to~$1$ in the cuspidal case already, and new ideas will be required already for non-relatively banal distinguished cuspidal representations. We will focus on the general case of distinguished irreducible~$\ell$-modular representations restricting to small rank in the paper \cite{KMS}. 
\\

Finally, we mention that this paper heavily relies on the results from \cite{AKMSS} and \cite{Secherre}, and can be seen as a natural continuation of the themes developed in these two papers. In particular our section on lifting distinction for cuspidal representations of finite general linear groups uses the same techniques as \cite{Secherre}, and the statements which we obtain here were known to the author of \cite{Secherre}.

\begin{acknowledgements}
It is a pleasure to thank Vincent S\'echerre for motivating discussions and very useful explanations concerning his paper~\cite{Secherre}.  We thank David Helm and Alberto M\'inguez for their interest and useful conversations.

This work has benefited from support from GDRI \emph{Representation Theory} 2016-2020 and the EPSRC grant EP/R009279/1, and the Heilbronn Institute for Mathematical Research. Parts of the paper were written while the second named author was at the conferences ``On the Langlands Program: Endoscopy and Beyond" held at the IMS of the National University of Singapore in 2018/2019 and ``Automorphic forms, automorphic representations and related topics"  held at the RIMS of the University of Kyoto in 2019. He would like to thank the organisers of both conferences, especially Wee Teck Gan and Shunsuke Yamana, and both institutions for their hospitality.
\end{acknowledgements}

\section{Notation}\label{sec:notation}

Let~$\F/\F_\so$ be a quadratic extension of non-archimedean local fields of odd residual characteristic~$p$.  For any finite extension~$\E/\F_\so$ we let~$|~|_\E$ be the absolute value,~$\val_\E$ the additive valuation,~$\oo_{\E}$ denote the ring of integers of~$\E$, with maximal ideal~$\pp_\E$, residue field~$\kk_\E$, and put~$q_\E=\#\kk_\E$.  We put~$|~|=|~|_\F$,~$\val=\val_\F$,~$\oo=\oo_{\F}$,~$\pp=\pp_\F$,~$\kk=\kk_\F$,~$q=q_\F$,~$|~|_\so=|~|_{\F_\so}$,~$\val_\so=\val_{\F_\so}$,~$\oo_\so=\oo_{\F_\so}$,~$\pp_{\so}=\pp_{\F_\so}$,~$\kk_\so=\kk_{\F_\so}$ and~$q_\so=q_{\F_\so}$.   

We let~$\ell$ denote a prime not equal to~$p$. Set~$\Ql$ to be an algebraic closure of the~$\ell$-adic numbers,~$\Zl$ its ring of integers, and~$\Fl$ its residue field.

Let~$\G$ be the~$\F$-points of a connected reductive group defined over~$\F$ and~$\mathcal{G}$ be the~$\kk$-points of a connected reductive group defined over~$\kk$.   

All representations considered are assumed to be smooth.  We consider representations of~$\G$ and~$\mathcal{G}$ and their subgroups on~$\Ql$ and~$\Fl$-vector spaces and the relations between them.  We let~$\R$ denote either~$\Ql$ or~$\Fl$, so that we can make statements valid in both cases more briefly.  By an~$\R$-representation we mean a representation on an~$\R$-vector space.  

An~$\R$-representation of~$\G$ or~$\mathcal{G}$ is called \emph{cuspidal} if it is irreducible and does not appear as a quotient of any representation parabolically induced from an irreducible representation of a proper Levi subgroup.  It is called \emph{supercuspidal} if it is irreducible and does not appear as a subquotient of any representation parabolically induced from an irreducible representation of a proper Levi subgroup.  Over~$\Ql$ a representation of~$\G$ or~$\mathcal{G}$ is cuspidal if and only if it is supercuspidal, however this is not the case over~$\Fl$, see \cite[III]{Vig96} and \cite{Kurinczuk} for examples of cuspidal non-supercuspidal representations.

\section{Background on integral representations and distinction}

\begin{definition}
Let~$\G$ be a locally profinite group and~$\H$ be a closed subgroup of~$\G$.  Let~$\pi$ be an~$\R$-representation of~$\G$ and~$\chi:\H\rightarrow \R^\times$ be a character.  We say that~$\pi$ is~\emph{$\chi$-distinguished} if~$\Hom_\H(\pi,\chi)\neq 0$.  We say that~$\pi$ is~\emph{distinguished} if it is~$1$-distinguished where~$1$ denotes the trivial character of~$\H$.
\end{definition}

We will mainly consider cases where $\H$ is the group of fixed points $\G^\sigma=\{g\in \G: \sigma(g)=g\}$ of an involution $\sigma$. In this case for any subset $\X\subset \G$, we set $\X^\sigma=\X\cap \G^\sigma$.  

\begin{definition}We call a triple~$(\G,\H,\sigma)$ \emph{an $\F$-symmetric pair} when 
\begin{enumerate}
\item~$\G=\mathbb{G}(\F)$ with~$\mathbb{G}$ a connected reductive group defined over~$\F$, and~$\sigma$ is an involution of~$\mathbb{G}$ defined over~$\F$;
\item~$\H$ is an open subgroup of the group~$\G^\sigma$.  
\end{enumerate} 
\end{definition}

The main symmetric pair of interest in this note will be~$(\GL_n(\F),\GL_n(\F_\so),\sigma)$, where~$\sigma$ is the involution induced from the nontrivial element of~$\Gal(\F/\F_\so)$. Two important basic results concerning this pair, that we shall use later are the following (\cite{FLickerDist}, \cite{PrasadDist} for~$\Ql$-representations, extended to~$\R$-representations in \cite[Theorem 3.1]{Secherre}):

\begin{proposition}\label{proposition p-adic mult 1 and ssduality}
Let $\pi$ be an irreducible $\R$-representation of $\GL_n(\F)$, then \[\dim(\Hom_{\GL_n(\F_\so)}(\pi,\R))\leqslant 1,\] moreover if this dimension 
is equal to one, then $\pi^\vee\simeq \pi^\sigma$.
\end{proposition}  

Let~$\K$ be a locally profinite group.  An irreducible~$\Ql$-representation~$\pi$ of~$\K$ is called \emph{integral} if it stabilises a~$\Zl$-lattice in its vector space. An integral irreducible~$\Ql$-representation~$\pi$ which stabilises a lattice~$\L$ induces an~$\Fl$-representation on the space $\L\otimes_{\Zl}\Fl$. When $\K$ is either a profinite group or the $F$-points of a connected reductive $\F$-group, the semisimplification~$r_{\ell}(\pi)=$ of $\L\otimes_{\Zl}\Fl$ is independent of the choice of~$\L$ and called the \emph{reduction modulo~$\ell$} of~$\pi$ (\cite[9.6]{Vig96} in the profinite setting where all representations are automatically defined over a finite extension of $\Fl$, or \cite[Theorem 1]{VigIntegral} in the context of reductive groups).  Given an irreducible~$\Fl$-representation~$\overline{\pi}$ of~$\K$, we will call any integral irreducible~$\Ql$-representation~$\pi$ of~$\K$ which satisfies~$r_{\ell}(\pi)=\overline{\pi}$ a \emph{lift} of~$\overline{\pi}$.

We shall see later that distinction of cuspidal representations of~$\G$ does not always lift, i.e. that an~$\ell$-modular cuspidal distinguished representation may have no distinguished lifts. However, we have the following general result which shows that distinction reduces modulo~$\ell$:

\begin{theorem}\label{thm distinction reduces mod ell}
Let~$(\G,\H,\sigma)$ be an $\F$-symmetric pair. Let~$\pi$ be an integral~$\ell$-adic supercuspidal representation if~$\G$, and let~$\chi$ be an integral character of~$\H$. Then 
if~$\pi$ is~$\chi$-distinguished, the representation~$r_{\ell}(\pi)$ is~$r_\ell(\chi)$-distinguished.
\end{theorem}
\begin{proof}
Note that~$\chi$ coincides with the central character of~$\pi$ restricted to~$\H$ which is also integral on~$\Z_\G\cap \H$ (where~$\Z_\G$ is the centre of~$\G$) hence we extend it to a character still denoted~$\chi$ to~$\Z_\G \H$:~$\chi(zh)=c_\pi(z)\chi(h)$.  Note that 
$\Hom_{\H}(\pi,\chi)=\Hom_{\Z_\G\H}(\pi,\chi)$. By \cite[Proposition 8.1]{KatoTakano} for~$\chi=1$, extended to general~$\chi$ in  \cite[Theorem 4.4]{Delorme}, for~$L$ a nonzero element of~$\Hom_{\H}(\pi,\chi)$, the map \[v\mapsto (g\mapsto L(\pi(g)v))\] embeds 
$\pi$ as a submodule of~$\ind_{\Z_\G \H}^\G(\chi)$. Now by \cite[Proposition II.3]{VigIntegral},~$\ind_{\Z_\G \H}^\G(\chi,\Zl)$ is an integral structure in~$\ind_{\Z_\G \H}^\G(\chi)$, hence its intersection~$\pi_e$ with~$\pi$ is an integral structure of~$\pi$ by \cite[9.3]{Vig96} (note that Vign\'eras works over a finite extension of $\Fl$, but her results apply here because both $\pi$ and 
$\chi$, hence both $\pi$ and $\ind_{\Z_\G \H}^\G(\chi)$ have $E$-structures by \cite[Section 4]{Vig96}). So~$\pi_e\subset \ind_{Z_\G \H}^\G(\chi,\Zl)$ but the map~$\Lambda:f\mapsto f(1_\G)$ is an element of 
$\Hom_{\H}(\ind_{\Z_\G \H}^\G(\chi,\Zl),\chi)$ which is nonzero on any submodule of~$\ind_{\Z_\G \H}^\G(\chi,\Zl)$, in particular on 
$\pi_e$. Up to multiplying~$\Lambda$ by an appropriate nonzero scalar in~$\Ql$, we can suppose that~$\Lambda(\pi_e)=\Zl$, and~$\Lambda$ induces a nonzero element of~$\Hom_{\H}(\pi_e\otimes \Fl,r_\ell(\chi))$. The result follows.
\end{proof}

\begin{remark}\label{remark finite dist reduces}
If $\K'$ is a closed subgroup of a profinite group $\K$, (smooth) finite dimensional $\Ql$-representations of $\K$ are always integral and the image of a lattice by a nonzero linear form on such a representation is obviously a lattice of $\Ql$, so the reduction modulo $\ell$ of a $(\K',\chi)$-distinguished finite dimensional $\Ql$-representation of $\K$ is $(\K',r_\ell(\chi))$-distinguished.
\end{remark}

\begin{remark}\label{remark reduction vs cuspidality}
The following observation sheds more light on Theorem \ref{thm distinction reduces mod ell} when $\G=\GL_n(\F)$. Let~$\pi$ be an integral supercuspidal~$\Ql$-representation of~$\GL_n(\F)$, then its reduction modulo~$\ell$ is (irreducible and) cuspidal, by \cite[III 4.25]{Vig96}.  This is however not true in general, see \cite{Kurinczuk} for an example of an integral supercuspidal~$\Ql$-representation whose reduction modulo~$\ell$ is reducible. 
\end{remark}

Let~$\K$ be a locally profinite group and~$\K'$ a closed subgroup.  While in general it appears a subtle question to ascertain when the distinction of~$\Fl$-representations of~$\K$ lifts, there is however one elementary case where it does: when the subgroup for which we want to study distinction~$\K'$ is profinite of pro-order prime to~$\ell$.  In this case, an $\ell$-modular finite dimensional  (smooth) representation of~$\K'$ is semisimple and reduction modulo~$\ell$ defines a bijection between the set of isomorphism classes of integral irreducible~$\Ql$-representations of~$\K'$ and the set of isomorphism classes irreducible~$\Fl$-representations of~$\K'$, and we have:

\begin{lemma}\label{lemma finite rel banal dist lifts}
Let~$\K$ be a locally profinite group and~$\K'$ be a compact subgroup of~$\K$.  Suppose that the pro-order of~$\K'$ is prime to~$\ell$.  Let~$\rho$ be an finite dimensional integral~$\Ql$-representation of~$\K$ and~$\chi$ be a character of~$\K'$.  Then~$\rho$ is~$\chi$-distinguished if and only if~$r_{\ell}(\rho)$ is~$r_{\ell}(\chi)$-distinguished.
\end{lemma}

\begin{remark}
If $\K$ is compact modulo centre, an irreducible $\Ql$-representation of $\K$ is always finite dimensional and is integral if and only if its central character is integral. 
\end{remark}

\section{Distinction for finite~$\GL_n$}\label{sec:finite}

For the rest of this section, we set~$\mathcal{G}=\GL_n(\kk)$, where (as before)~$\kk$ denotes a finite field of odd cardinality~$q$.  If~$\kk/\kk_\so$ is a quadratic extension of~$\kk_\so$ we denote by~$\sigma$ the non trivial element of~$\Gal(\kk/\kk_\so)$ 
and set $\mathcal{G}_\so=\GL_n(\kk_\so)$. 

We recall the definitions of~self-dual and~$\sigma$-self-dual representations of~$\mathcal{G}$:
\begin{definition}
\begin{enumerate}
\item Suppose~$\kk/\kk_\so$ is a quadratic extension of finite fields, then a representation~$\rho$ of~$\mathcal{G}$, over~$\Ql$ or~$\Fl$, is called~\emph{$\sigma$-self-dual} if~$\rho^\sigma\simeq\rho^\vee$.
\item A representation~$\rho$ of~$\mathcal{G}$, over~$\Ql$ or~$\Fl$, is called \emph{self-dual} if~$\rho\simeq \rho^\vee$.  
\end{enumerate}
\end{definition}

\subsection{Basic results on distinction}

The following multiplicity one results are \cite[Remark 3.2 with the adhoc modification in the proof of Theorem 3.1, Proposition 6.10 and Remark 6.11]{Secherre}:

\begin{proposition}\label{prop finite multiplicity 1}
Let $\rho$ be an irreducible $\R$-representation of $\mathcal{G}$:
\begin{enumerate}
\item If $\kk/\kk_\so$ is a quadratic extension of finite fields, then $\dim(\Hom_{\mathcal{G}_\so}(\rho,\chi))\leq 1$ 
for any character $\chi$ of $\mathcal{G}_\so$.
\item If $\rho$ is cuspidal and $r$ and $s$ are two nonnegative integers such $r+s=n\geq 2$. Then 
$\dim(\Hom_{(\GL_r\times \GL_s)(\kk)}(\rho,\chi))\leq 1$ for any character $\chi$ of $\GL_r\times \GL_s$ and this dimension is equal to zero if $r$ and $s$ are positive and $r\neq s$.
\end{enumerate}
\end{proposition}

The final goal of this section is to understand when a cuspidal~$\Fl$-representation of a finite general linear group which is distinguished by a maximal Levi subgroup or by a Galois involution has a lift which does not share the same distinction property.

The connection between~$(\sigma)$-self-dual representations and distinction comes from:

\begin{lemma}\label{lemma finite distinction vs s(s)duality}
\begin{enumerate}
\item\label{lemma14ii}  A~$\GL_{n}(\kk_\so)$-distinguished irreducible~$\R$-representation of $\GL_n(\kk)$ is~$\sigma$-self-dual.  Moreover if~$\rho$ is supercuspidal, we have an equivalence:~$\rho$ is $\sigma$~self-dual if and only if it is~$\GL_{n}(\kk_\so)$-distinguished.
\item\label{lemma14i} A~supercuspidal representation of $\GL_n(\kk)$ is~self-dual if and only if either $n=1$ and it a quadratic character, or if $n$ is even and it is~$(\GL_{n/2}\times\GL_{n/2})(\kk)$-distinguished.
\end{enumerate}
\end{lemma}
\begin{proof}
The first assertion of (i) follows from \cite[Remark 3.2]{Secherre}, and the second from \cite[Lemma 8.3]{Secherre}. The second assertion follows from \cite[Lemmas 7.1 and 7.3]{Secherre}.
\end{proof}

\subsection{Self-dual and~$\sigma$-self-dual cuspidal representations via the Green--Dipper--James parametrisation}\label{section DJ param}
In this subsection either~$\kk$ is an arbitrary finite field and we consider self-dual representations of~$\GL_n(\kk)$ or~$\kk/\kk_\so$ is quadratic and we consider~$\sigma$-self-dual representations of~$\GL_n(\kk)$ where~$\langle\sigma\rangle=\Gal(\kk/\kk_\so)$.

Let~$\ll/\kk$ be a degree~$n$ extension of~$\kk$.  A character~$\theta:\ll^\times\rightarrow \Ql^\times$ is called \emph{$\kk$-regular} if~$\#\{\theta^\tau:\tau\in\Gal(\ll/\kk)\}=n$, i.e. the orbit of~$\theta$ under~$\Gal(\ll/\kk)$ has maximal sise. By \cite{Green}, there is a surjective map
\begin{align*}
\{\kk\text{-regular characters of }\ll^\times\rightarrow\Ql^\times\}&\rightarrow \{\text{supercuspidal~$\Ql$-representations of }\mathcal{G}\}/\simeq\\
\theta&\mapsto \rho(\theta),
\end{align*}
The character formula given in \cite{Green} also implies:
\begin{enumerate}
\item Two such cuspidal representations~$\rho(\theta)$ and~$\rho(\theta')$ are isomorphic if and only if there exists~$\tau\in\Gal(\ll/\kk)$ such that~$\theta'=\theta^\tau$.  
\item The dual~$\rho(\theta)^\vee$ is isomorphic to~$\rho(\theta^{-1})$.
\item If~$\kk/\kk'$ is a finite extension and~$\tau\in \Gal(\ll/\kk')$, we have~$\rho(\theta^\tau)\simeq\rho(\theta)^\tau$. 
\end{enumerate}  

The following is well-known, and a similar proof to ours can be found in \cite[Lemmas 7.1 \& 8.1]{Secherre} in the greater generality of supercuspidal~$\R$-representations, we provide a proof as a warm-up:

\begin{lemma}\label{lemma existence of s or ssdual ell-adic scusp}
\begin{enumerate}
\item If there exists a~$\sigma$-self-dual supercuspidal~$\Ql$-rep\-resentation of~$\mathcal{G}$, then~$n$ is odd.
\item If there exists a self-dual supercuspidal~$\Ql$-representation of~$\mathcal{G}$, then~$n$ is either one or even.
\end{enumerate}
\end{lemma}

\begin{proof}
\begin{enumerate}
\item Suppose that~$\rho$ is a~$\sigma$-self-dual cuspidal~$\Ql$-representation, and write~$\rho=\rho(\theta)$ for a $\kk$-regular character~$\theta$.  Choose an extension of~$\sigma$ to~$\widetilde{\sigma}\in\Gal(\ll/\kk_\so)$.
Then as~$\rho(\theta^{-1})\simeq\rho(\theta)^\vee\simeq\rho(\theta)^\sigma$, necessarily~$\theta^{\widetilde{\sigma}}=(\theta^{-1})^\tau$ for some~$\tau\in \Gal(\ll/\kk)$. Hence~$(\tau^{-1}\circ \widetilde{\sigma})^2$ fixes~$\theta$, so it is~$1$ as~$\theta$ is regular. This implies that~$\tau^{-1}\circ \widetilde{\sigma}$ is a~$\kk_\so$-linear involution of~$\ll$ which extends~$\sigma$. However the cyclic group~$\Gal(\ll/\kk_\so)$ contains a unique element of order~$2$. If~$n$ was even,~$\tau^{-1}\circ \widetilde{\sigma}$ would belong to~$\Gal(\ll/\kk)$ and this is absurd as it extends~$\sigma$ which is not~$\kk$-linear.
\item Suppose that~$\rho$ is a self-dual cuspidal~$\Ql$-representation, and write~$\rho=\rho(\theta)$ for a $\kk$-regular character~$\theta$. In this case, reasoning as before, necessarily~$\theta=\tau(\theta^{-1})$ for some~$\tau\in \Gal(\ll/\kk)$ and it follows that~$\tau^2=1$.  Either~$\tau=1$ hence~$\theta^2=1$, but there is a unique non-trivial quadratic character of 
$\ll^\times$ which is thus fixed by all~$\tau\in \Gal(\ll/\kk)$ and the trivial character of $\ll^\times$ is also $\Gal(\ll/\kk)$-invariant, so we deduce that~$n=1$ as~$\theta$ is regular.  Or~$\tau$ has order~$2$ hence~$n=\#\Gal(\ll/\kk)$ is even.
\end{enumerate}
\end{proof}

We now recall the classification of cuspidal~$\Fl$-representations of James \cite{James}.  We have a surjective map
\begin{align*}
\{\text{supercuspidal~$\Ql$-representations of }\mathcal{G}\}/\simeq&\rightarrow \{\text{cuspidal~$\Fl$-representations of }\mathcal{G}\}/\simeq\\
\rho(\theta)&\mapsto \overline{\rho(\theta)}
\end{align*}
given by reduction modulo~$\ell$.  

Given a character~$\theta:\ll^\times\rightarrow\Ql^\times$ we can uniquely write~$\theta=\theta_{r}\theta_s$ with~$\theta_{r}$ of order prime to~$\ell$ and~$\theta_s$ of order a power of~$\ell$. The parametrisation of James enjoys the following properties:
\begin{enumerate}
\item Two supercuspidal~$\Ql$-representations~$\rho(\theta),\rho(\theta')$ have isomorphic reductions modulo~$\ell$ if and only if there exists~$\tau\in\Gal(\ll/\kk)$ such that~$\theta'_{r}=\theta_{r}^\tau$.  
\item~$\overline{\rho(\theta)}$ is supercuspidal if and only if~$\theta_r$ is regular.
\end{enumerate}

\subsection{$\sigma$-self-dual lifts of cuspidal~$\Fl$-representations}\label{section ssdual finite lifts}
We now specialise to the case~$\kk/\kk_\so$ is quadratic. Write~$\Gamma=\Hom(\ll^\times,\Ql^\times)$, then~$\Gamma=\Gamma_s\times\Gamma_r$ where~$\Gamma_s$ consists of the characters of~$\ell$-power order, and~$\Gamma_{r}$ consists of the characters with order prime to~$\ell$.  

We study~$\sigma$-self-dual lifts of cuspidal~$\Fl$-representations, and when~$n$ is even there are no~$\sigma$-self-dual supercuspidal~$\Ql$-representations by Lemma \ref{lemma existence of s or ssdual ell-adic scusp}.  Hence without loss of generality, we can assume that~$n$ is odd. Moreover as reduction modulo~$\ell$ commutes with taking contragredients and with the action of~$\sigma$, this implies that when the cuspidal representation~$\overline{\rho}$ of~$\mathcal{G}$ is not~$\sigma$-self-dual, it has no 
$\sigma$-self-dual lifts, so we suppose that it is from now on.

For~$\gamma\in\Gamma$ we set~$\Gal(\ll/\kk)_\gamma=\{\tau\in\Gal(\ll/\kk):\gamma^\tau=\gamma\}$.  Let~$\theta\in\Gamma$, we can decompose~$\theta=\theta_r\theta_s$, and we have~$\Gal(\ll/\kk)_\theta=\Gal(\ll/\kk)_{\theta_r}\cap \Gal(\ll/\kk)_{\theta_s}$. In particular,~$\theta$ is regular if and only if~$\Gal(\ll/\kk)_{\theta_r}\cap \Gal(\ll/\kk)_{\theta_s}=\{1\}.$ 

Let~$\ll_\so$, be the unique subextension of~$\ll/\kk_\so$ of degree~$n$ as an extension of~$\kk_\so$ and put~$\Gamma_\so=\Hom(\ll_\so^\times,\Ql^\times)$. We have an embedding
\begin{align*}
i:\Gamma_\so\hookrightarrow\Gamma,\qquad i:\gamma\mapsto \gamma\circ \N_{\ll/\ll_\so},
\end{align*}
by surjectivity of the norm.  Hence~$\Gamma^+=i(\Gamma_\so)$ is unique subgroup of the cyclic group~$\Gamma$ of order~$q_\so^n-1$. Write~$\widetilde{\sigma}$ for the unique involution in~$\Gal(\ll/\kk_\so)$, which extends~$\sigma$ (as~$n$ is odd).   By Hilbert's theorem 90, we have
\[\Gamma^+=\{\gamma\in \Gamma:\gamma^{\widetilde{\sigma}}=\gamma\}.\]
On the other hand, the unique subgroup of the cyclic group~$\Gamma$ of order~$q_\so^n+1$ is~
\begin{align*}
\Gamma^-&=\{\gamma \in\Gamma:\gamma\circ \N_{\ll/\ll_\so}=1\}=\{\gamma\in\Gamma:\gamma^{\widetilde{\sigma}}=\gamma^{-1}\},
\end{align*}
as the norm map is surjective.  Notice that $(q_\so^n+1,q_\so^n-1)=2$ because~$q$ is odd, so~$\Gamma^+\cap\Gamma^-=\{1,\eta\}$, where~$\eta$ denotes the unique quadratic character in~$\Gamma$. Moreover if~$\ell$ is odd:
\begin{enumerate}
\item  If~$\ell\mid q_\so^{n}-1$, then~$\Gamma_s\subseteq\Gamma^+$;
\item If~$\ell\mid q_\so^{n}+1$, then~$\Gamma_s\subseteq \Gamma^-$.
\end{enumerate}

Before giving the full solution of the lifting~$\sigma$-self-duality for~$\ell$-modular cuspidal representations, 
we characterise~$\ell$-modular cuspidal~$\sigma$-self-duality in terms of the Dipper and James parametrisation. 

\begin{proposition}\label{characterisation ssduality}
Let~$\overline{\rho}$ be a cuspidal representation of~$\mathcal{G}$ and suppose that~$n$ and~$\ell$ are odd, then~$\overline{\rho}$ is~$\sigma$-self-dual 
if and only if~$\theta_r^{\widetilde{\sigma}}=\theta_r^{-1}$.
\end{proposition}
\begin{proof}
Write~$\overline{\rho}=\overline{\rho(\theta)}$ for a $\kk$-regular character~$\theta$, and let~$\widetilde{\sigma}\in\Gal(\ll/\kk_\so)$ be the unique involution extending~$\sigma$.  One implication is obvious, for the other we thus suppose that~$\overline{\rho(\theta)}$ is~$\sigma$-self-dual. Then there 
exists~$\tau\in \Gal(\ll/\kk)$ such that~$\theta_r^{\widetilde{\sigma}\tau}=\theta_r^{-1}$. This implies that~$\tau^2=(\widetilde{\sigma}\tau)^2$ belongs to~$\Gal(\ll/\kk)_{\theta_r}$. On the other hand the order 
of~$\tau$ is odd because $n$ is, hence~$\tau$ as well belongs to~$\Gal(\ll/\kk)_{\theta_r}$, so $\theta_r^{\widetilde{\sigma}}=\theta_r^{-1}$.
\end{proof}

We have the following complete result when~$\ell$ is odd.

\begin{proposition}\label{Propsigmaself-dualcases}
Assume that~$n$ and~$\ell$ are odd. Let~$\overline{\rho}$ be a~$\sigma$-self-dual cuspidal~$\Fl$-representation of $\mathcal{G}$.
\begin{enumerate} 
\item\label{Prop18parti} Suppose that~$\ell$ is prime with~$q^{n}-1$. Then the unique supercuspidal lift of~$\overline{\rho}$ is~$\sigma$-self-dual.
\item\label{Prop18partii} Suppose that~$\ell\mid q_\so^{n}-1$. 
\begin{enumerate}
\item  If~$\overline{\rho}$ is supercuspidal and~$\ell^a$ is the highest power of~$\ell$ dividing~$q^n-1$, then there is a unique~$\sigma$-self-dual supercuspidal lift amongst the~$\ell^a$ supercuspidal lifts of~$\overline{\rho}$.  In terms, of Green's parameterisation of supercuspidal~$\Ql$-representations, if~$\rho(\theta)$ is a lift of~$\overline{\rho}$ then~$\rho(\theta_r)$ is the unique~$\sigma$-self-dual supercuspidal lift of~$\overline{\rho}$.
\item If~$\overline{\rho}$ is cuspidal non-supercuspidal, then none of its supercuspidal lifts are~$\sigma$-self-dual.
\end{enumerate}
\item\label{Prop18partiii} Suppose that~$\ell\mid q_\so^{n}+1$. Then all supercuspidal lifts of~$\overline{\rho}$ are~$\sigma$-self-dual.
\end{enumerate}
\end{proposition}
\begin{proof}
Write~$\overline{\rho}=\overline{\rho(\theta)}$ for a $\kk$-regular character~$\theta$ of $\ll^\times$, and let~$\widetilde{\sigma}\in\Gal(\ll/\kk_\so)$ be the unique involution extending~$\sigma$. 

The set of isomorphism classes of supercuspidal~$\Ql$-representations lifting~$\overline{\rho}$ is then
\[\{\rho(\theta_r\mu):\mu\in\Gamma_s,\theta_r\mu \ \kk-\text{regular}\}/\simeq.\]
Such a representation~$\rho(\theta_r\mu)$ is~$\sigma$-self-dual if and only if there exists~$\tau\in\Gal(\ll/\kk)$ such that~$(\theta_r\mu)^{\widetilde{\sigma}}=(\theta_r^{-1}\mu^{-1})^\tau$.  As~$\theta_r\mu$ is regular, this condition implies that~$\tau^2=(\widetilde{\sigma}\tau)^2$ is the identity, so that~$\tau=\Id$ as~$n$ is odd. So~$\rho(\theta_r\mu)$ is~$\sigma$-self-dual if and only if~$\theta_r^{\widetilde{\sigma}}=\theta_r^{-1}$ and~$\mu^{\widetilde{\sigma}}=\mu^{-1}$, and the set of 
$\sigma$-self-dual lifts of~$\overline{\rho}$ is equal to 
\[\{\rho(\theta_r\mu):\mu\in\Gamma_s\cap \Gamma^-,\theta_r\mu \ \kk-\text{regular}\}/\simeq\] because 
the condition~$\theta_r^{\widetilde{\sigma}}=\theta_r^{-1}$ is always satisfied thanks to Proposition \ref{characterisation ssduality}. In particular when~$\theta_r$ is regular, then all~$\theta_r\mu$ must be regular as well and 
the cardinality of the set of $\sigma$-selfdual lifts of~$\overline{\rho}$ is that of~$\Gamma_s$, namely the highest power of~$\ell$ dividing~$q^n-1$.

In particular if~$\ell$ is prime to~$q^n-1$ then~$\Gamma_s$ is trivial and this proves \ref{Prop18parti}. 
 
 If~$\ell\mid q_\so^n-1$.  Then~$\Gamma_s\subseteq \Gamma^+$, and~$\Gamma_s\cap \Gamma^-=\Gamma_s\cap \Gamma^-\cap\Gamma^+=\{1\}$ because~$\Gamma^+\cap\Gamma^-=\{1,\eta\}$ and~$\eta\not\in\Gamma_s$.  Hence if~$\overline{\rho}$ is supercuspidal i.e. if~$\theta_r$ is regular, then~$\rho(\theta_r)$ is the unique~$\sigma$-self-dual supercuspidal lift of~$\overline{\rho}$ whereas if~$\overline{\rho}$ is cuspidal non-supercuspidal, then it has no~$\sigma$-self-dual supercuspidal lift, and we have shown \ref{Prop18partii}.
 
Finally suppose that~$\ell\mid q_\so^n+1$, then~$\Gamma_s\subseteq \Gamma^-$ so all supercuspidal lifts of~$\overline{\rho}$ are~$\sigma$-self-dual, and we have shown \ref{Prop18partiii}.
\end{proof}

In the case~$\ell=2$ we have:

\begin{proposition}\label{Prop always non ssdual lift when l=2} 
Assume that~$n$ is odd and~$\ell=2$. Let~$\overline{\rho}$ be a~$\sigma$-self-dual cuspidal~$\Fl$-representation of $\mathcal{G}$, then it has 
a non~$\sigma$-self-dual lift.
\end{proposition}
\begin{proof}
Write~$\overline{\rho}=\overline{\rho(\theta)}$ for a regular character~$\theta$, and let~$\widetilde{\sigma}\in\Gal(\ll/\kk_\so)$ be the unique involution extending~$\sigma$. 

First note that~$q^n-1=(q_\so^n-1)(q_\so^n+1)$ so the highest power of~$2$ dividing~$q^n-1$ which is the order of~$\Gamma_s$ 
does not divide~$q_\so^n+1$ as~$2$ also divides~$q_\so^n-1$. In particular~$\Gamma_s$ is not a subgroup of 
$\Gamma^-$, so if~$\mu_0$ is a generator of~$\Gamma_s$, then~$\widetilde{\sigma}(\mu_0)\neq \mu_0^{-1}$. Now we claim that 
$\rho(\theta_r\mu_0)$ is a non~$\sigma$-self-dual lift of~$\overline{\rho(\theta)}$. First it is supercuspidal: indeed 
$\Gal(\ll/\kk)_{\mu_0}\subset \Gal(\ll/\kk)_{\theta_s}$ because~$\theta_s$ is a power of~$\mu_0$, hence 
$\Gal(\ll/\kk)_{\theta_r}\cap \Gal(\ll/\kk)_{\mu_0}$ is trivial because~$\Gal(\ll/\kk)_{\theta_r}\cap \Gal(\ll/\kk)_{\theta_s}$ is. Moreover suppose that~$\rho(\theta_r\mu_0)$ was~$\sigma$-self-dual, then following the beginning of the proof of Proposition \ref{Propsigmaself-dualcases}, this would imply that both~$\theta_r$ and~$\mu_0$ belong to~$\Gamma^-$, which is absurd.
\end{proof}

\subsection{Self-dual lifts of self-dual cuspidal~$\Fl$-representations}\label{section sdual finite lifts}
If there exists a self-dual supercuspidal~$\Ql$-representation of~$\mathcal{G}$ then~$n$ is one or even by Lemma \ref{lemma existence of s or ssdual ell-adic scusp}.  The case~$n=1$ is straightforward, a character is self-dual if and only if it is quadratic and we treat it separately:
\begin{proposition}\label{prop n=1propfinitefield}
Suppose that~$n=1$.  Then~$1,\eta$ are the unique self-dual supercuspidal~$\Ql$-representations of~$\GL_1(\kk)$.  The reductions~$\overline{1},\overline{\eta}$ of~$1,\eta$ respectively are the unique self-dual cuspidal~$\Fl$-representations of~$\GL_1(\kk)$.  
\begin{enumerate}
\item Suppose that~$\ell\nmid q-1$. Then~$\overline{1},\overline{\eta}$ have~$1,\eta$ respectively as unique lift.
\item Suppose that~$\ell\mid q-1$ and let~$\ell^a$ be the highest power of~$\ell$ dividing~$q-1$.  Then~$\overline{1},\overline{\eta}$ each have~$\ell^a$-supercuspidal lifts of which~$1,\eta$ (respectively) is the unique self-dual supercuspidal lift.
\end{enumerate}
Note that case (ii) contains the case~$\ell=2$, in which case~$\overline{1}=\overline{\eta}$. So in particular in case (ii) non trivial lifts of the trivial character of $\k^\times$ always exist. 
\end{proposition}

Hence for the rest of this section we assume that~$n$ is even.  Let~$\sigma'$ denote the unique involution in~$\Gal(\ll/\kk)$ and~$\ll_{\so}'=\ll^{\sigma'}$ denote the~$\sigma'$-fixed subfield.  Then we have an embedding
\[i':\Hom(\Gal(\ll_\so'/\kk),\Ql^\times)\hookrightarrow \Gamma,\qquad i':\gamma\mapsto\gamma\circ \N_{\ll/\ll_\so'},\]
by surjectivity of the norm so its image is the unique subgroup of~$\Gamma$ of order~$q^{n/2}-1$: \[\Gamma_+=\{\gamma\in\Gamma:\gamma^{\sigma'}=\gamma\}.\]
The unique subgroup of~$\Gamma$ of order~$q^{n/2}+1$ is thus~\[\Gamma_-=\{\gamma\in\Gamma:\gamma^{\sigma'}=\gamma^{-1}\},\] their intersection given by~$\Gamma_+\cap\Gamma_-=\{1,\eta\}$ as~$q$ is odd.  As~$(q^{n/2}+1,q^{n/2}-1)=2$, we deduce that if~$\ell$ is odd:
\begin{enumerate}
\item  If~$\ell\mid q^{n/2}-1$, then~$\Gamma_s\subseteq\Gamma_+$;
\item If~$\ell\mid q^{n/2}+1$, then~$\Gamma_s\subseteq \Gamma_-$.
\end{enumerate}

The results concerning self-duality look very similar to those concerning~$\sigma$-self-duality, however in one case we only consider lifting of distinction.

\begin{proposition}\label{self-dualfinitefieldprop}
Suppose that~$n=2m\geqslant 2$ is even and that~$\ell$ is odd. 
Let~$\overline{\rho}$ be a~self-dual cuspidal~$\Fl$-representation of $\mathcal{G}$. 
\begin{enumerate} 
\item\label{Prop110parti} Suppose that~$\ell$ is prime with~$q^{n}-1$. Then the unique supercuspidal lift of~$\overline{\rho}$ is~self-dual.
\item\label{Prop110partii} Suppose that~$\ell\mid q^{n/2}-1$. 
\begin{enumerate}
\item \label{Prop110partiia} If~$\overline{\rho}$ is supercuspidal and~$\ell^a$ is the highest power of~$\ell$ dividing~$q^n-1$, then there is a unique~self-dual lift of~$\overline{\rho}$ amongst its~$\ell^a$ supercuspidal lifts.  In terms, of Green's parameterisation of supercuspidal~$\Ql$-representations, if~$\rho(\theta)$ is a lift of~$\overline{\rho}$ then~$\rho(\theta_r)$ is the unique~self-dual supercuspidal lift of~$\overline{\rho}$.
\item\label{Prop110partiib} If~$\overline{\rho}$ is cuspidal non-supercuspidal, then none of its supercuspidal lifts are self-dual.
\end{enumerate}
\item\label{Prop110partiii} Suppose that~$\ell\mid q^{n/2}+1$ and that~$\overline{\rho}$ is $\GL_m(\kk)\times \GL_m(\kk)$-distinguished. Then all supercuspidal lifts of~$\overline{\rho}$ are~self-dual.
\end{enumerate}
\end{proposition}

\begin{proof}
 
A lift~$\rho(\theta_r\mu)$ with~$\mu\in \Gamma_s$ is~self-dual if and only if there exists~$\tau\in\Gal(\ll/\kk)$ such that~$\theta_r\mu=(\theta_r^{-1}\mu^{-1})^\tau$ hence~$\tau^2=\Id$, i.e.~$\tau=\Id$ or~$\sigma'$. The first case is impossible 
as it would imply that~$\theta_r\mu$ is the quadratic character in~$\Gamma$, which would contradict its regularity. Hence 
$\rho(\theta_r\mu)$ is self-dual if and only if~$\theta_r^{\sigma'}=\theta_r^{-1}$ and~$\mu^{\sigma'}=\mu^{-1}$.
Thus, if~$\theta_r^{\sigma'}=\theta_r^{-1}$, then the set \[\{\rho(\theta_r\mu):\mu\in\Gamma_s\cap \Gamma_-,\theta_r\mu\text{ regular}\}/\simeq \]
is a full set of representatives for the isomorphism classes of~self-dual lifts of~$\overline{\rho}$ (and there are no self-dual lifts if~$\theta_r^{\sigma'}\neq \theta_r^{-1}$).

For Parts \ref{Prop110parti}, \ref{Prop110partiia} as~$\theta_r$ is regular, the self-duality of $\overline{\rho}$ implies that~$\theta_r^{\sigma'}=\theta_r^{-1}$. Hence 
Parts \ref{Prop110parti}, \ref{Prop110partiia} follow in the same way as their analogues in Proposition \ref{Propsigmaself-dualcases}. Part \ref{Prop110partiib} is obvious as~$\Gamma_s\cap \Gamma_-$ is trivial in this case. 
Finally \ref{Prop110partiii} holds because distinction lifts in this case by Lemma \ref{lemma finite rel banal dist lifts}.

\end{proof}

When~$\ell=2$, we have the exact analogue of Proposition \ref{Prop always non ssdual lift when l=2} with the same proof, 
replacing~$q_\so^n$ by~$q^{n/2}$. 

\begin{proposition}\label{Prop always non sdual lift when l=2} 
Assume that~$n\geqslant 2$ is even and~$\ell=2$. Let~$\overline{\rho}$ be a self-dual cuspidal~$\Fl$-representation of $\mathcal{G}$, then it has a non self-dual lift.
\end{proposition}

\section{Type theory and distinction}

From now on we set~$\G=\GL_n(\F)$,~$\G_\so=\GL_n(\F_\so)$, and~$\sigma$ the Galois involution. We use the Bushnell--Kutzko construction of cuspidal representations of $\G$ \cite{BK}, extended by Vign\'eras to the setting of cuspidal~$\R$-representations \cite[\S III]{Vig96}. We summarise its properties that we will use, and refer the reader to \cite{BK} and \cite{Vig96} for more details on this construction.

\subsection{Properties of types}

Let $\pi$ be a cuspidal $\R$-representation of $\G$.  Then associated to it is a family of explicitly constructed pairs $(\bJ,\bl)$, called \emph{extended maximal simple types}, where $\bJ$ is a compact open subgroup of~$\G$ containing the centre $\Z_\G$ of~$\G$ with $\bJ/\Z_\G$ compact, and
$\bl$ is an irreducible (hence finite dimensional) representation of $\bJ$ such that \[\pi\simeq \ind_{\bJ}^{\G}(\bl).\] 
We abbreviate extended maximal simple type to \emph{type} for the rest of the paper, and will say~\emph{$\R$-type} when we wish to specify the field~$\R$ considered.  
%
%

Let~$(\bJ,\bl)$ be an~$\R$-type \emph{in}~$\pi$, i.e.~associated to~$\pi$ as described above.  Types enjoy the following key properties:
\begin{enumerate}[(T-1)]
\item Two types in~$\pi$ are conjugate in~$\G$, \cite[6.2.4]{BK} and \cite[III 5.3]{Vig96}
\item The group $\bJ$ has a unique maximal compact subgroup $\J$, and $\J$ has a unique maximal normal pro-$p$-subgroup $\J^1$, cf.~\cite[\S 3.1]{BK} for the definitions of these groups. 
\item\label{Edefinition}There is a subfield~$\E$ of~$\mathcal{M}_n(\F)$ containing~$\F$, the multiplicative group of which normalises~$\J$, and~$\bJ=\E^\times\J$.  (In fact, we have summarised this construction in reverse; the extension~$\E/\F$ is part of the original data used to construct the type.)   The quotient $\J/\J^1$ is isomorphic to $\GL_m(\kk_E)$ with $m=n/[\E:\F]$. 
 Moreover,~$\E^\times \cap \J=\oo_{\E}^\times$, hence~$\bJ=\langle \w_\E \rangle \J$ is the semi-direct product
~of~$\J$ with the group generated by~$\w_\E$.
\item\label{type urtwist} Let~$(\bJ,\bl)$ be a type, let $\bl'$ be a representation of $\bJ$, set $\pi=\ind_{\bJ}^{\G}(\bl)$ and $\pi'=\ind_{\bJ}^{\G}(\bl')$. If $\bl'\mid_{\J}\simeq \bl\mid_{\J}$, then $(\bJ,\bl')$ is a type and the cuspidal representation $\pi'$ is an unramified twist of $\pi$. Conversely if $(\bJ,\bl')$ is a type and the cuspidal representation $\pi'$ is an unramified twist of $\pi$, then $\bl'\mid_{\J}\simeq \bl\mid_{\J}$.
\item The representation $\bl$ (by construction) decomposes (non-uniquely) as a tensor product $\bk\otimes \bt$, where:
\begin{itemize} 
\item $\bt$ is a representation of $\bJ$ trivial on $\J^1$ which restricts irreducibly to $\J$, and the representation 
of $\J/\J^1$ induced by $\bt$ identifies with a cuspidal representation of $\GL_m(\kk_\E)$.
\item $\bk$ is a representation of $\bJ$ which restricts irreducibly to $\J^1$.
\end{itemize}
\item The representation $\pi$ is supercuspidal if and only if $\bt$ induces a supercuspidal restriction on~$\J/\J^1$, \cite[III 5.14]{Vig96}.
\item\label{type bt unique when bt fixed} The pair $(\bJ,\bk\otimes \bt')$ is another type in $\pi$ if and only if $\bt\simeq \bt'$.
\item When~$\R=\Ql$, the representation $\pi$ is integral if and only if $\bl$ is integral.
\item The construction is compatible with reduction modulo~$\ell$ in the following sense: given a~$\Ql$-type~$(\bJ,\bl)$ with~$\bl$ integral,~$(\bJ,r_{\ell}(\bl))$ is an~$\Fl$-type and $r_{\ell}(\ind_{\bJ}^\G(\bl))=\ind_{\bJ}^\G(r_{\ell}(\bl))$ is a cuspidal~$\Fl$-representation, \cite[III 4.25]{Vig96}. 
\item \label{type lift} The construction lifts \cite[III 4.29]{Vig96}: given an~$\Fl$-type~$(\bJ,\bk\otimes\bt)$, there is a unique irreducible~$\Ql$-representation~$\widetilde{\eta}$ of~$\J^1$ which lifts~$\bk\mid_{\J^1}$ and we can fix a extension~$\widetilde{\bk}$ of~$\widetilde{\eta}$ with~$r_{\ell}(\widetilde{\bk})=\bk$.  As all of the extensions of~$\widetilde{\eta}$ to~$\bJ$ are related by twisting by a character trivial on~$\J^1$, Property (T-\ref{type bt unique when bt fixed}) implies that the set of isomorphism classes of lifts of~$\pi=\ind_{\bJ}^\G(\bk\otimes \bt)$ is in bijection with the set of isomorphism classes of lifts of~$\bt$  by~$\widetilde{\bt}\mapsto \ind_{\bJ}^\G(\widetilde{\bk}\otimes\widetilde{\bt})$.
\item We call a field extension~$\E/\F$ associated to a type in~$\pi$ as in (T-~\ref{Edefinition}) a \emph{parameter field} for~$\pi$.  While there are potentially many choices, the ramification index~$e(\E/\F)$, inertial degree~$f(\E/\F)$, and (hence) the degree~$[\E:\F]$ are invariants of~$\pi$ as follows from \cite[3.5.1]{BK}. As such, we write~$d(\pi)=[\E:\F]$,~$e(\pi)=e(\E/\F)$ and~$f(\pi)=f(\E/\F)$.  We write~$m(\pi)$ for~$n/d(\pi)$.  These invariants are compatible with reduction modulo~$\ell$: for~$\pi$ an integral cuspidal~$\Ql$-representation we have
\[d(\pi)=d(r_{\ell}(\pi)),~e(\pi)=e(r_{\ell}(\pi)),~f(\pi)=f(r_{\ell}(\pi)),~m(\pi)=m(r_{\ell}(\pi)).\]
\end{enumerate} 

\subsection{Galois-self-dual types}\label{GSDTsection}

It has recently been showed in \cite{AKMSS} and \cite{Secherre} that the construction of types
also enjoys good compatibility properties with $\sigma$-self-duality and distinction. Indeed, according to \cite[\S 4]{AKMSS}, 
if $\pi$ is a cuspidal $\R$-representation of $\G$ which is $\sigma$-self-dual, then one can chose a type $(\bJ,\bl)$ in~$\pi$ such that:

\begin{enumerate}[(SSDT-I)]
\item\label{ssdtpe1} $\bJ$ (hence $\J$ and $\J^1$) and $\E$ are $\sigma$-stable, and $\bl^\vee\simeq \bl^\sigma$.
\item\label{ssdtpe2} Set~$\E_\so=\E^\sigma$, then $\E/\E_\so$ is a quadratic extension and we can choose a uniformiser~$\w_\E$ with~$\sigma(\w_\E)=\w_\E$ if~$\E/\E_\so$ is unramified and $\sigma(\w_E)=-\w_E$ if~$\E/\E_\so$ ramified as in \cite[(5.2)]{Secherre}.
\item\label{ssdtpe refined} $\bk$ hence $\bt$ are $\sigma$-self-dual (\cite[Lemma 8.9]{Secherre}).
\end{enumerate}

The ramification index~$e(\E/\E_\so)\in\{1,\ 2\}$ (and is equal to $1$ if $\F/\F_\so$ is unramified) is an invariant of~$\pi$ and we write
\[e_\sigma(\pi)=e(\E/\E_\so).\] 

\begin{remark}\label{remark E vs T}
This latter invariant is also equal to the ramification index of the extension $\T/\T^\sigma$, where 
$\T$ is the maximal tamely ramified extension of $\F$ contained in $\E$ thanks to \cite[Remark 4.15 (2)]{Secherre}. We use this fact when referring to some results of \cite{Secherre}.
\end{remark}

In \cite[\S 6.2]{AKMSS}, another invariant, a positive integer~$e_\so(\pi)$ dividing~$n$ defined only in terms of the~$\sigma$-stable group~$\bJ$ is associated to~$\pi$. By \cite[Lemma 5.10]{AKMSS}, we have the following description:
\begin{equation*}
e_\so(\pi)=\begin{cases}
2 e(\E_\so/\F_\so)&\text{if }e_\sigma(\pi)=2\text{ and }m(\pi)\neq 1;\\
e(\E_\so/\F_\so)&\text{otherwise.}
\end{cases}
\end{equation*}

 Again, these invariants are compatible with reduction modulo~$\ell$: for~$\pi$ an integral~$\sigma$-self-dual supercuspidal~$\Ql$-representation we have
\[e_\sigma(\pi)=e_{\sigma}(r_{\ell}(\pi)),~e_\so(\pi)=e_\so(r_{\ell}(\pi)).\]

\begin{definition}\label{definition ssdtype} We call a type $\bl=\bk\otimes \bt$ satisfying Conditions (SSDT-\ref{ssdtpe1}) to 
(SSDT-\ref{ssdtpe refined}) a \textit{$\sigma$-self-dual type}.
\end{definition}

Note that an immediate consequence of the existence of a $\sigma$-self-dual cuspidal type in a~$\sigma$-self-dual cuspidal representation is that in some cases there are no $\sigma$-self-dual cuspidal representations, as follows from \cite[Lemma 6.9 and Lemma 8.1]{Secherre}:

\begin{lemma}\label{lemma no ssdual sc reps in bad cases}
Let $\pi$ be $\sigma$-self-dual cuspidal~$\R$-representation.
\begin{enumerate}
\item If $e_\sigma(\pi)=1$ and $\pi$ is supercuspidal, then $m(\pi)$ is odd.
\item If $e_\sigma(\pi)=2$, then $m(\pi)$ is either equal to $1$ or even.
\end{enumerate} 
\end{lemma}

A crucial property of $\sigma$-self-dual types is the following:

\begin{proposition}{\cite[Lemma 5.19]{Secherre}}\label{prop ssdtpe canonical iso}
Let~$(\bJ,\bl)$ be a~$\sigma$-self-dual type. Then there exists a unique character $\chi_{\bk}$ of $\bJ^\sigma$ trivial on $(\J^1)^\sigma$ such that 
\[\Hom_{(\J^1)^{\sigma}}(\bk,\R)=\Hom_{{\bJ}^\sigma}(\bk,\chi_{\bk}),\] and the canonical map 
\[\Hom_{{\bJ}^\sigma}(\bk,\chi_{\bk})\otimes \Hom_{{\bJ}^\sigma}(\bt,\chi_{\bk}^{-1})\rightarrow 
\Hom_{{\bJ}^\sigma}(\bl,\R)\] is an isomorphism.
\end{proposition}

In many cases, it is shown in \cite{Secherre} that one can choose $\chi_{\bk}=1$ above, including the supercuspidal case:

\begin{proposition}\label{prop existence of dist kappa}
Let $\pi$ be a $\sigma$-self-dual supercuspidal $\R$-representation and $(\bJ,\bl)$ be a $\sigma$-self-dual type of $\pi$, then 
one can choose $\bk$ such that $\chi_{\bk}=1$. 
\end{proposition}
\begin{proof}
The only cases to consider are those which are not ruled out by Lemma \ref{lemma no ssdual sc reps in bad cases}, and the assertion then follows from \cite[Propositions 6.15 and 8.10]{Secherre}.
\end{proof}

\begin{remark}\label{remark dist of kappa reduces}
Note that if $\R=\Ql$ above, then $\pi$ is integral and $(\bJ,r_\ell(\bt))$ is a $\sigma$-self-dual type for $r_\ell(\pi)$. Moreover $\chi_{r_\ell(\bk)}=1$ if $\chi_{\bk}=1$. Indeed, thanks to Remark \ref{remark finite dist reduces} applied to 
$(\bJ/\F_\so^\times,\bk)$, the representation $r_\ell(\bk)$ is distinguished, hence $\chi_{r_\ell(\bk)}=1$ by the first part of Proposition \ref{prop ssdtpe canonical iso}.
\end{remark}

\subsection{Generic types and distinguished types}

There are in general more than one $\G_\so$-conjugacy class of $\sigma$-self-dual types in a $\sigma$-self-dual cuspidal $\R$-representation (see \cite[Section 1.11]{Secherre}). However there is only one $\G_\so$-conjugacy class among those which contain a  generic type in the following sense. 

We denote by~$\N$ be the maximal unipotent subgroup of the subgroup of upper triangular matrices in~$\G$, and~$\N_\so=\N^\sigma$. Let $\psi$ a nondegenerate character of $\N$. Note that such a character is always integral with nondegenerate reduction modulo $\ell$ as $\N$ is exhausted by its pro-$p$-subgroups.

\begin{definition} Let $(\bJ,\bl)$ be an $\R$-type, we say that $(\bJ,\bl)$ is a \textit{$\psi$-generic type} if 
\[\Hom_{\N\cap \bJ}(\bl,\psi)\neq \{0\}.\] We say that it is \emph{generic} if it is $\psi$-generic for some nondegenerate character of $\N$.
\end{definition}

\begin{remark}\label{remark twist of generic still generic}
Note that if $\mu$ is a character of $\G$ and $(\bJ,\bl)$ is a $\psi$-generic type, then $(\bJ,\mu\mid_{\bJ}\otimes \bl)$ is also 
$\psi$-generic.
\end{remark}

We will also use the following observation later.

\begin{lemma}\label{lemma lift of generic types}
If $(\bJ,\bl)$ is an integral $\Ql$-type and $\psi$ be a nondegenerate character of $\N$. If $(r_\ell(\bJ),r_\ell(\bl))$ is $r_\ell(\psi)$-generic, then 
$(\bJ,\bl)$ is $\psi$-generic.
\end{lemma}
\begin{proof}
It is a consequence of Lemma \ref{lemma finite rel banal dist lifts}, once one observes that $\bJ\cap \N$ is a (compact) pro-$p$ group.
\end{proof}

If the type we consider is moreover $\sigma$-self-dual, we will only consider distinguished nondegenerate characters $\psi$ of $\N$, i.e. those which are trivial on $\N_\so$:

\begin{definition}
A $\sigma$-self-dual $\R$-type is called \emph{generic} if it is $\psi$-generic with respect to a distinguished nondegenerate character $\psi$ of $\N$.
\end{definition}

Our definition of a \textit{generic $\sigma$-self-dual type} coincides with the definition given in \cite[Definition 9.1]{Secherre} (see the discussion after \cite[Definition 5.7]{AKMSS}). There are two fundamental facts about these types, first they always occur in $\sigma$-self-dual cuspidal representations.

\begin{proposition}{\cite[Proposition 5.5]{AKMSS}}\label{prop existence of generic ssdual type}
Let $\pi$ be a $\sigma$-self-dual cuspidal $\R$-representation of $\G$ and let $\psi$ be a distinguished nondegenerate character of $\N$, then $\pi$ has a $\psi$-generic $\sigma$-self-dual type, which is moreover unique up to $\N_\so$-conjugacy.
\end{proposition}

The second one concerns distinguished representations, which are $\sigma$-self-dual thanks to Proposition \ref{proposition p-adic mult 1 and ssduality}.

\begin{theorem}{\cite[Corollary 6.6]{AKMSS}, \cite[Theorem 9.3]{Secherre}}\label{thm dist type} Let $\pi$ be a $\sigma$-self-dual cuspidal $\R$-representation and $(\bJ,\bl)$ a generic $\sigma$-self-dual type of $\pi$. Then $\pi$ is distinguished if and only if $\bl$ is $\bJ^\sigma$-distinguished. 
\end{theorem}

\begin{definition}\label{definition dtype} We call a $\sigma$-self-dual type $(\bJ,\bl)$ a \textit{distinguished type} if~$\bl$ is~$\bJ^\sigma$-distinguished.
\end{definition}

We have the following surprising result, which completes Theorem \ref{thm dist type} and is evidence of the interplay between genericity and Galois distinction for~$\GL_n(\F)$:

\begin{lemma}\label{lemma distinguished implies generic}
A distinguished $\R$-type is automatically $\sigma$-self-dual generic. 
\end{lemma}
\begin{proof}
Take a $\sigma$-self-dual 
type $(\bJ,\bl)$ such that $\bl$ is $\bJ^\sigma$-distinguished. Then $\ind_{\bJ}^{\G}(\bl)$ is distinguished by Mackey theory and Frobenius reciprocity. But then by \cite[Remark 6.7]{AKMSS} (and the equivalence of our definition of generic type with that given in \cite[Defintion 9.1]{Secherre}), the type $(\bJ,\bl)$ must be $\psi$-generic for some character distinguished nondegenerate character $\psi$ of $\N$.
\end{proof}

We end this section with the following important corollary of Proposition \ref{prop ssdtpe canonical iso}:

\begin{corollary}\label{cor of canonical iso} A $\sigma$-self-dual type $(\bJ,\bl)$ is a distinguished type 
if and only if $\Hom_{{\bJ}^\sigma}(\bk,\chi_{\bk})$ and $\Hom_{{\bJ}^\sigma}(\bt,\chi_{\bk}^{-1})$ are nonzero, in which case both are one dimensional.
\end{corollary} 

\subsection{The relative torsion group of a distinguished representation} 

For the rest of this section we fix $\pi$ a distinguished cuspidal $\R$-representation and $(\bJ,\bl)$ a distinguished type in~$\pi$. 
We set~$\w_{\E_\so}=\w_{\E}^2$ if~$\E/\E_\so$ is ramified and~$\w_{\E_\so}=\w_{\E}$ if~$\E/\E_\so$ is unramified. When~$e_\sigma(\pi)=2$ and~$m(\pi)=2r$ is even, we denote by~$w$ the element of~$\bJ$ corresponding to~$\left(\begin{smallmatrix}0&1_r\\1_r&0\end{smallmatrix}\right)$ as in \cite[Lemma 6.19]{Secherre}.  We define the \textit{relative torsion group of $\pi$} to be the following group:
\[\X_\so(\pi)=\{\mu_\so \in \Hom(\G_\so,\R^\times): \mu_\so \ \text{is unramified}, \ \Hom_{\G_\so}(\pi,\mu_\so)\neq \{0\}\}.\]

\begin{theorem}\label{thm Xso} 
Let~$\pi$ be a cuspidal distinguished~$\R$-representation of~$\G$ and set $\w'=\w_Ew$ if ~$e_\sigma(\pi)=2$ and~$m(\pi)$ is even and 
$\w'=\w_{\E_\so}$ otherwise. Then we have:
\begin{enumerate}
\item Let~$\mu_\so$ be an unramified character of~$\G_\so$, then~$\mu_\so\in \X_\so(\pi)$ if and only if $\mu_\so(\w')=1$. 
\item Let $\chi_\so$ be an unramified character of $\F_\so^\times$, then $\chi_\so\circ \det \in \X_\so(\pi)$ if and only if~$\chi_\so(\w_\so)^{n/e_\so(\pi)}=1$.
\end{enumerate}
\end{theorem}
\begin{proof}
Note that if~$e_\sigma(\pi)=2$, then~$m(\pi)$ is even or equal to~$1$ thanks to Lemma \ref{lemma no ssdual sc reps in bad cases}. 
Then~$\bJ^\sigma$ 
is generated by~$\w'$ and~$\J^\sigma$, thanks to \cite[Lemmas 6.18, 6.19 and 8.7]{Secherre}. Let $\mu_\so$ be an unramified character of $\G_\so$ and denote by $\mu$ an unramified extension of it to 
$\G$, hence $\mu_\so\in \X_\so(\pi)$ if and only if $\mu\otimes \pi$ is distinguished. Suppose that 
$\mu\otimes \pi$ is distinguished, then $(\bJ,\bk \otimes (\mu\otimes \bt))$ is a $\sigma$-self-dual type which is in fact a distinguished type thanks to Remark \ref{remark twist of generic still generic} and Theorem \ref{thm dist type}, and conversely if 
$(\bJ,\bk \otimes (\mu\otimes \bt))$ is a distinguished type, then $\mu\otimes \pi$ is distinguished. So 
$\mu_\so\in \X_\so(\pi)$ if and only if $(\bJ,\bk \otimes (\mu\otimes \bt))$ is a distinguished type, which is if and only if 
$\Hom_{\bJ^\sigma}(\mu\otimes\bt,\chi_{\bk}^{-1})$ has dimension~$1$ according to Corollary \ref{cor of canonical iso}.

However~$\Hom_{\bJ^\sigma}(\bt,\chi_{\bk}^{-1})$ has dimension~$1$ by the same corollary, but $\Hom_{\J^\sigma}(\bt,\chi_{\bk}^{-1})$ is already at most one dimensional thanks to Proposition \ref{prop finite multiplicity 1}, so from these multiplicity one statements we deduce that $\mu_\so\in \X_\so(\pi)$ if and only if $\Hom_{\bJ^\sigma}(\mu\otimes\bt,\chi_{\bk}^{-1})=\Hom_{\bJ^\sigma}(\bt,\chi_{\bk}^{-1})$. Finally this translates as: $\mu_\so\in \X_\so(\pi)$ if and only if ~$\mu(\w')=1$ and proves (i).

Now let~$\chi_\so$ be an unramified chareter of $\F_\so^\times$ and let $\chi$ be an unramified character of $\F^\times$ extending it. If $e_\sigma(\pi)=2$ then $\F/\F_\so$ is ramified according to \cite[Lemma 4.14]{Secherre}.  

Suppose that $e_\sigma(\pi)=2$ and ~$m(\pi)$ is even,  
we have: 
\[\chi_\so(\det(w'))=\chi(\det(\w_E))=\chi(\N_{\E /\F }(\w_{\E}))^{m(\pi)}=\chi(\w)^{f(\E/F)m(\pi)}=\chi(\w)^{n/e(\E/\F)}\]
\[=\chi(\w)^{ne(\F/\F_\so)/e_\sigma(\pi)e(\E_\so/\F_\so)}=\chi(\w)^{n/e(\E_\so/\F_\so)}=\chi(\w_\so)^{n/2e(\E_\so/\F_\so)}= 
\chi(\w_\so)^{n/e_\so(\pi)}\] thanks to 
\cite[Lemma 5.10]{AKMSS}. 
Otherwise we have:
\[\chi_\so(\det(w'))=\chi_\so(\N_{\E /\F}(\w_{E_\so}))^{m(\pi)}=\chi_\so(\N_{\E_\so /\F_\so}(\w_{E_\so}))^{m(\pi)}=\chi_\so(\w_\so)^{f(\E_\so/\F_\so)m(\pi)}\]
\[=\chi_\so(\w_\so)^{n/e(\E_\so/\F_\so)}=\chi(\w_\so)^{n/e_\so(\pi)}\] thanks to 
\cite[Lemma 5.10]{AKMSS} again. 
\end{proof}

It has the following two corollaries.

\begin{corollary}\label{cor cardinality of Xso}
Let~$\pi$ be a distinguished cuspidal~$\R$-representation of~$\G$. Write~$n/e_\so(\pi)=a(\pi)\ell^r$ with~$a(\pi)$ prime to~$\ell$. Then~$\X_\so(\pi)$ is a cyclic group, of order~$n/e_\so(\pi)$ if $\R=\Ql$, and of order~$a(\pi)$ if $\R=\Fl$.
\end{corollary}

\begin{corollary}\label{cor reduction mod ell on Xso}
Let~$\pi$ be a distinguished cuspidal (hence integral)~$\Ql$-representation of~$\G$. Then the homomorphism  
\[r_\ell:\mu_\so\mapsto r_\ell(\mu_\so)\] is surjective from $\X_\so(\pi)$ to $\X_\so(r_\ell(\pi))$, and its kernel is 
the $\ell$-singular part of $\X_\so(\pi)$.
\end{corollary}
\begin{proof}
It suffices to verify the assertion on the kernel. It is clear that the $\ell$-singular part of $\X_\so(\pi)$ belongs to the kernel of $r_\ell$. Conversely if $r_\ell(\mu_\so)=1$, write $\mu_\so=(\mu_\so)_r(\mu_\so)_s$ with $(\mu_\so)_r$ of order prime to $\ell$ and $(\mu_\so)_s$ of order a power of $\ell$, then $r_\ell((\mu_\so)_r)=r_\ell(\mu_\so)=1$ so $(\mu_\so)_r=1$ because 
$r_\ell$ induces a bijection between the group of roots of unity of order prime to $\ell$ in $\Ql^\times$ and 
the group of roots of unity in $\Fl^\times$.
\end{proof}

\section{Relatively banal cuspidal representations of~$p$-adic~$\GL_n$}\label{section relatively bana reps}

In \cite{MSDuke} and \cite{MSbanal}, M\'inguez and S\'echerre single out a class of irreducible representations called \emph{banal} for which the Zelevinski classification works particularly nicely.  For cuspidal representations, the following definition can be given (\cite[Remarque 8.15]{MSDuke} and \cite[Lemme 5.3]{MSDuke}).
\begin{definition}
A cuspidal~$\Fl$-representation~$\pi$ is called \emph{banal} if~$q^{n/e(\pi)}\not\equiv 1[\ell]$.  
\end{definition}

The following definition is new and is motivated by our cuspidal~$\L$-factor computation later and an analogy with banal cuspidal representations and the Rankin--Selberg computation of \cite{KMNagoya}.  We show in Section \ref{section comparison banal rel banal} how it is a natural analogue of banal for the symmetric pair~$(\G,\G_\so,\sigma)$.

\begin{definition}\label{definition relbanaldef}
Let~$\pi$ be a distinguished cuspidal~$\Fl$-representation. We say that it is \emph{relatively banal} if~$q_\so^{n/e_\so(\pi)}\not \equiv 1[\ell]$.
\end{definition}

Theorem \ref{thm Xso} (ii) has the following third consequence:

\begin{corollary}[of Theorem \ref{thm Xso}]\label{cor relatively banal simple definition}
Let~$\pi$ be a distinguished cuspidal~$\Fl$-representation, it is relatively banal if and only if 
it is not $|\det(\ \cdot \ )|_{\so}$-distinguished.
\end{corollary}

Before stating the next lemma, we make the following observation which shows that
the statement of the lemma in question (Lemma 6.6) is indeed complete.

\begin{remark}\label{remark CRASremark}
For any character~$\chi$ of~$\G_\so$, there are no~$\chi$-distinguished cuspidal~$\R$-representations~$\pi$ of~$\G$ with~$e_\sigma(\pi)=2$ when~$m(\pi)\geqslant 3$ is odd, by Proposition \ref{prop ssdtpe canonical iso} and \cite[Lemma 6.9]{Secherre}. 
\end{remark}
While we have defined relatively banal distinguished representations in terms of the invariant~$e_\so(~)$, we will use the following equivalent formulation:

\begin{lemma}\label{comparison}
Let~$\pi$ be a~$\sigma$-self-dual cuspidal~$\R$-representation of~$\G$.  Let~$\E$ be a~$\sigma$-self-dual parameter field for~$\pi$.  Then
\begin{enumerate}
\item\label{comparisoni} If~$e_\sigma(\pi)=1$, then~$q_\so^{n/e_\so(\pi)}=q_{\E_\so}^{m(\pi)}$ (and is also equal to~$q_{\so}^{n/e(\pi)}$ if~$\F/\F_\so$ is unramified and~$q^{n/2e(\pi)}$ if~$\F/\F_\so$ is ramified).
\item\label{comparisonii} If~$e_\sigma(\pi)=2$ and~$m(\pi)=1$, then~$q_\so^{n/e_\so(\pi)}=q_{\E_\so}$ (and is also equal to $q^{n/e(\pi)}=q_{\E}$). 
\item\label{comparisoniii}  If~$e_\sigma(\pi)=2$ and~$m(\pi)\geqslant 2$ is even, then~$q_\so^{n/e_\so(\pi)}=q_{\E_\so}^{m(\pi)/2}$ (and is also equal to $q^{n/2e(\pi)}=q_\E^{m(\pi)/2}$).
\end{enumerate}
\end{lemma}

\begin{proof}
In all cases, we have~$q^{n/e(\pi)}=q_\E^{m(\pi)}$. 
In case \ref{comparisoni},~$q_{\E_\so}^m$ is the positive square root of~$q_\E^m$. However by \cite[Lemma 5.10]{AKMSS}, we have 
\[e_\so(\pi)= e(\E_\so/\F_\so)=e(\E/\F_\so)/e_\sigma(\pi)=e(\E/\F_\so)=e(\E/\F) e(\F/\F_\so)=e(\pi)e(\F/\F_\so).\]
If~$\F/\F_\so$ is unramified, then~$q_\so^{n/e_\so(\pi)}=q_\so^{n/e(\pi)}$ is the positive square root of~$q^{n/e(\pi)}$. Now if~$\F/\F_\so$ is ramified, then~$q_\so^{n/e_\so(\pi)}=q^{n/2e(\pi)}_\so$ and \ref{comparisoni} is proved. 
 
 In case \ref{comparisonii} by \cite[Lemma 5.10]{AKMSS} again, we have
 \[e_\so(\pi)= e(\E_\so/\F_\so)=e(\E/\F_\so)/e_\sigma(\pi)=e(\E/\F_\so)/2=e(\E/\F) e(\F/\F_\so)/2=e(\pi)e(\F/\F_\so)/2.\] 
 However~$e(\F/\F_\so)=2$ by \cite[Lemma 4.14]{Secherre} so~$e_\so(\pi)= e(\pi)$ and~$q=q_\so$ which proves case \ref{comparisonii}.
 
 Finally in case \ref{comparisoniii} by \cite[Lemma 5.10]{AKMSS} again, we have
 \[e_\so(\pi)= 2e(\E_\so/\F_\so)= 2e(\E/\F_\so)/e_\sigma(\pi)=e(\E/\F_\so)=e(\E/\F) e(\F/\F_\so)=e(\pi)e(\F/\F_\so),\] and 
~$e(\F/\F_\so)=2$ by \cite[Lemma 4.14]{Secherre} so~$e_\so(\pi)= 2e(\pi)$ and~$q=q_\so$ which proves case \ref{comparisoniii}.
\end{proof}

Immediately, from Remark \ref{remark CRASremark} and Lemma \ref{comparison}, we have:

\begin{corollary}
A banal distinguished cuspidal~$\Fl$-representation of~$\G$ is relatively banal.
\end{corollary}

\begin{remark}
A banal cuspidal~$\Fl$-representation is supercuspidal.  However, there are relatively banal distinguished cuspidal non-supercuspidal~$\Fl$-representations. For example when~$n=3$ and $\ell\neq 2$, the non-normalised parabolic induction of the trivial representation of the Borel subgroup has a cuspidal subquotient~$\St(3)$ when~$q_\so^3\equiv-1[\ell]$, and when~$\F/\F_\so$ is unramified it is relatively banal distinguished (see \cite{KMS}). 
\end{remark}

Before proving the main result of this section, it will be useful to know that there are no relatively banal distinguished cuspidal representations~$\pi$ when~$e_\sigma(\pi)=1$ and~$m(\pi)$ is even:

\begin{lemma}\label{lemma no relatively banal in bad parity}
Let~$\pi$ be a cuspidal~$\Fl$-representation of~$\G$ which is~$\sigma$-self-dual. Suppose that~$e_\sigma(\pi)=1$, that~$m(\pi)$ is even, and~$q_\so^{n/e_\so(\pi)}\not\equiv 1[\ell]$, then~$\pi$ is not distinguished.
\end{lemma}
\begin{proof}
Let~$(\bJ,\bk\otimes\bt)$ be a~$\sigma$-self-dual generic~$\Fl$-type for~$\pi$ (Proposition \ref{prop existence of generic ssdual type}) with~$\sigma$-stable parameter field~$\E$.  Suppose~$\pi$ is distinguished.  Then by Theorem \ref{thm dist type} we can suppose that~$\bk\otimes \bt$ is distinguished as well. By Proposition \ref{prop ssdtpe canonical iso},~$\bt$ is~$\chi_{\bk}^{-1}$-distinguished hence 
$\rho=\bt\mid_{\J}$ seen as a representation of~$\GL_{m(\pi)}(\kk_\E)$ is~$\chi_{\bk}^{-1}$-distinguished by the group~$\GL_{m(\pi)}(\kk_{\E_\so})$, i.e. 
that~$\rho'=\chi \otimes \rho$ is distinguished for an extension~$\chi$ of~$\chi_{\bk}$ to~$\kk_{\E}^\times$. Now 
by Lemma \ref{comparison} and Lemma \ref{lemma finite rel banal dist lifts},~$\rho'$ has a distinguished lift, which contradicts Lemma \ref{lemma existence of s or ssdual ell-adic scusp}. 
\end{proof}

\begin{remark}
Notice that the statement of Lemma \ref{lemma no relatively banal in bad parity} is not empty as~$\sigma$-self-dual representations~$\pi$ exist under the hypothesis~$e_\sigma(\pi)=1$ and~$q_\so^{n/e_\so(\pi)}\not \equiv 1[\ell]$: for example when~$n=2$ and~$\F/\F_\so$ is unramified the non-normalised parabolic induction of the trivial representation of the Borel subgroup has a cuspidal subquotient~$\St(2)$ when~$q\equiv-1[\ell]$ which is~$\sigma$-self-dual and~$e_\sigma(\St(2))=1$ as~$\F/\F_\so$ is unramified.
\end{remark}

\begin{lemma}\label{lemma retriction to J and distinction}
Let~$(\bJ,\bl)$ be an~$\R$-type such that~$\bJ$ is~$\sigma$-stable and put~$\pi=\ind_{\bJ}^{\G}(\bl)$.  If~$\bl\mid_\J$ is distinguished, then~$\pi$ is the unramified twist of a distinguished representation.  Conversely suppose that moreover~$\bl=\bk\otimes \bt$ is generic and that $\bk$ is distinguished and $\sigma$-self-dual, if~$\pi$ is the unramified twist of a distinguished representation, then~$\bt\mid_{\J}$ is distinguished.
\end{lemma}
\begin{proof}
If~$\bl\mid_{\J}$ is distinguished, then we can extend~$\bl$ to a distinguished representation~$\bl_\F$ of~$\F^\times\J$ by setting~$\bl_\F(\w_\F)=1$.  The induced representation~$\ind_{\F^\times\J}^{\E^\times \J}(\bl_\F)$ is distinguished, and because~$\bJ/\F^\times\J\simeq \langle \w_\E \rangle/ \langle \w_\F \rangle$ is cyclic, all of its irreducible subquotients extend~$\bl_\F$ by Clifford theory, so one extension~$\bl_\E$ of~$\bl_\F$ to~$\bJ$ is distinguished.  Hence~$\ind_{\bJ}^\G(\bl_\E)$ is distinguished and an unramified twist of~$\ind_{\bJ}^\G(\bl)$ by Property (T-\ref{type urtwist}).

For the partial converse, by twisting by an unramified character without loss of generality we can suppose that~$\pi$ is distinguished (and~$\bk$ is the same).  Then~$\bt$ is $\sigma$-self-dual thanks to 
Property (T-\ref{type bt unique when bt fixed}) hence $\bl$ as well, and it is distinguished because of Theorem \ref{thm dist type}. Then Proposition \ref{prop ssdtpe canonical iso} implies that $\bt$, hence~$\bt\mid_\J$ is distinguished.
\end{proof}

Relatively banal distinguished cuspidal~$\Fl$-representations enjoy very nice lifting properties:

\begin{theorem}\label{thm relatively banal distinction lift}
Let~$\pi$ be a cuspidal and distinguished~$\Fl$-representation of~$\G$.
\begin{enumerate} 
\item\label{liftingi} Then~$\pi$ is relatively banal if and only if all of its lifts are unramified twists of distinguished representations. 
\item\label{liftingii} If it is relatively banal, then it has a distinguished lift.
\end{enumerate}
\end{theorem}
\begin{proof}
Suppose that~$\pi$ is relatively banal distinguished.  Choose a~distinguished type~$(\bJ,\bl)$ in~$\pi$ and let~$\widetilde{\pi}$ be a lift of~$\pi$. We can choose a type in~$\widetilde{\pi}$ of the form~$(\bJ,\widetilde{\bl})$ with~$r_\ell(\widetilde{\bl})=\bl$ by property (T-\ref{type lift}).  As~$\ell$ is coprime to~$\J^\sigma$, we can apply Lemma \ref{lemma finite rel banal dist lifts} and~$\widetilde{\bl}\mid_{\J}$ is distinguished because so is~$\bl\mid_\J$.  Hence~$\widetilde{\pi}$ is a unramified twist of a distinguished representation by Lemma \ref{lemma retriction to J and distinction} and this proves one implication in \ref{liftingi}.

We now prove \ref{liftingii}. Suppose that $\pi$ is relatively banal. By the implication already proved in \ref{liftingi} we know that~$\pi$ has 
a lift $\widetilde{\pi}$ which is $\widetilde{\mu_\so}$-distinguished for $\widetilde{\mu_\so}$ an unramified character of $\G_\so$. Let 
$\widetilde{\mu}$ be an unramified character of $\G$ extending $\widetilde{\mu_\so}$, then $\widetilde{\mu}^{-1}\otimes \widetilde{\pi}$ is distinguished. However, because $\pi$ is distinguished, setting $\mu=r_\ell(\widetilde{\mu})$, the representation $\mu^{-1}\otimes \pi$ is $\mu_\so^{-1}$-distinguished for $\mu_\so=r_\ell(\widetilde{\mu_\so})=\mu\mid_{\G_\so}$. Thanks to Corollary \ref{cor reduction mod ell on Xso}, $\mu_\so$ has a lift $\widetilde{\mu_\so}'\in \X_\so(\widetilde{\mu}^{-1}\otimes \widetilde{\pi})$. Writing $\mu=\chi\circ\det$, it is possible to extend $\widetilde{\mu_\so'}$ to an unramified character $\widetilde{\mu'}$ of $\G$ such that if $\mu'=r_\ell(\mu)=\chi'\circ \det$, then $\chi'(\w)=\chi(\w)$: indeed as $\mu$ and $\mu'$ both extend $\mu_\so$, this is automatic if $\F/\F_\so$ is unramified, whereas if $\F/\F_\so$ is ramified $\chi'(\w)=\pm \chi(\w)$ is automatic, and we can always change $\widetilde{\mu'}$ so that this sign is $+$. With such choices, the representation  $\widetilde{\mu'}\widetilde{\mu}^{-1}\otimes \widetilde{\pi}$ is a distinguished lift of $\pi$.

It remains to prove the second implication of \ref{liftingi}. Suppose that~$\pi$ is not relatively banal, i.e.~$q_\so^{n/e_\so(\pi)}\equiv 1[\ell]$.  Suppose, for the sake of contradiction, that all lifts of~$\pi$ are distinguished up to an unramified twist and let~$\widetilde{\pi}$ be a lift of~$\pi$. Under this assumption the argument used to prove \ref{liftingii} shows that $\pi$ has a distinguished supercuspidal lift $\widetilde{\pi}$.
 This lift has a distinguished type $(\bJ,\widetilde{\bk}\otimes \widetilde{\bt})$ with $\chi_{\widetilde{\bk}}=1$ thanks to Theorem \ref{thm dist type} and Proposition \ref{prop existence of dist kappa}, and we set $\bk=r_\ell(\widetilde{\bk})$ and $\bt=r_\ell(\widetilde{\bt})$ so in particular $(\bJ,\bk\otimes \bt)$ is a distinguished type thanks to Remark \ref{remark finite dist reduces} (and also Lemma  \ref{lemma distinguished implies generic}). Proposition \ref{prop ssdtpe canonical iso} together with Lemma \ref{lemma existence of s or ssdual ell-adic scusp} imply that if $e_\sigma(\pi)=2$, then either $m(\pi)=1$ or it is even, and if $e_\sigma(\pi)=1$, then $m(\pi)$ is odd. Then $\bt\mid_\J$ is distinguished according to 
 Remark \ref{remark finite dist reduces}, but the assumption $q_\so^{n/e_\so(\pi)}\equiv 1[\ell]$ translated in terms of 
 $\GL_{m(\pi)}(\kk_\E)$ thanks 
 to Lemma \ref{comparison} together with Propositions \ref{Propsigmaself-dualcases} (ii), \ref{Prop always non ssdual lift when l=2}, \ref{prop n=1propfinitefield} (ii), \ref{Propsigmaself-dualcases} (ii) and \ref{Prop always non sdual lift when l=2} imply that $\bt\mid_\J$ has a non distinguished lift $\widetilde{\bt'}$. This lift extends to $\langle \w_\so \rangle \J$ to a lift of $\bt\mid_{\langle \w_\so \rangle\J}$ by setting $\widetilde{\bt'}(\w_\so)=1$. Then by Clifford theory, because the quotient $\bJ/\langle \w_\so \rangle\J\simeq \langle \w_\E \rangle/\langle \w_\so \rangle$ is cyclic, the representation $\ind_{\langle \w_\so \rangle \J}^{\bJ}(\widetilde{\bt'})$ contains 
 a lift of $\bt$ which extends $\widetilde{\bt'}$, and we again denote it $\widetilde{\bt'}$. Then the representation $\pi'=\ind_{\bJ}^{\G}(\widetilde{\bk}\otimes \widetilde{\bt}')$ is then a supercuspidal lift of $\pi$. As  $\widetilde{\bk}\otimes \widetilde{\bt}'$ reduces to the generic type $\bk\otimes \bt$, it is generic thanks to Lemma \ref{lemma lift of generic types}, hence it can't be an unramified twist of a distinguished representation according to the second part of Lemma \ref{lemma retriction to J and distinction}.
\end{proof}

\begin{remark}\label{equivalence of dist by ur and ur twist of dist}
Note that as an unramified character of 
$\G_\so$ always has an unramified extension to $\G$, Part (i) of Theorem \ref{thm relatively banal distinction lift} can also be stated as: $\pi$ 
is relatively banal if and only if all its lifts are distinguished by an unramified character. 
\end{remark}

\section{Asai~$\L$-factors of cuspidal representations}\label{section Asai}
\subsection{Asai~$\L$-factors}\label{section def Asai} 
Let~$\N$ be the maximal unipotent subgroup of the subgroup of upper triangular matrices in~$\G$, and~$\N_\so=\N^\sigma$.  Let~$\psi$ be a non-degenerate~$\R$-valued 
character of~$\N$ trivial on~$\N_\so$.  
Let~$\pi$ be an~$\R$-representation of~$\G$ of \emph{Whittaker type} (i.e. of finite length with a one dimensional space of Whittaker functionals) 
with Whittaker model~$\Ww(\pi,\psi)$. We refer to \cite[Section 2]{KM17} for  more details as well as basic facts about Whittaker functions and their 
analytic behaviour. For~$\W\in\Ww(\pi,\psi)$ and~$\Phi\in\Cc_c^\infty(\F_\so^n)$ and~$l\in\mathbb{Z}$ define the local \emph{Asai coefficient} to be
\begin{equation}
\label{Asaiintegral}
\I^l_\As(\X,\Phi,\W)= \int_{\substack{\N_{\so}\backslash \G_{\so}\\\mathrm{val}(\det(g))=l}} \W(g)
\Phi(\eta_n g) \ dg,
\end{equation}
where~$\eta_n$ denotes the row vector~$\left(\begin{smallmatrix}0&\cdots&0&1\end{smallmatrix}\right)$ and~$dg$ denotes a right~invariant measure on~$\N_\so\backslash \G_\so$ with values in~$\R$. We refer the reader to \cite[Section 2.2]{KM17} for details on~$\R$-valued equivariant measures on homogeneous spaces and their properties. The integrand in the Asai coefficient has compact support so it is well-defined and it moreover vanishes for~$l<<0$.  We define the \emph{Asai integral} of~$\W\in\Ww(\pi,\psi)$ and~$\Phi\in\Cc_c^\infty(\F_\so^n)$ to be the formal Laurent series
\begin{equation}
\I_\As(\X,\Phi,\W)=\sum_{l\in\mathbb{Z}}\I^l_\As(\X,\Phi,\W)\X^l.\end{equation}
In exactly the same way as in \cite[Theorem 3.5]{KM17}, we deduce the following lemma:

\begin{lemma}
For~$\W\in\Ww(\pi,\psi)$ and~$\Phi\in\Cc_c^\infty(\F_\so^n)$,~$\I_\As(\X,\Phi,\W)\in \R(\X)$ is a rational function. 
Moreover, as~$\W$ varies in~$\Ww(\pi,\psi)$ and~$\Phi$ varies in~$\Cc_c^\infty(\F_\so^n)$ these functions generate a~$\R[\X^{\pm 1}]$-fractional 
ideal of~$\R(\X)$ independent of the choice of~$\psi$.
\end{lemma}

In the setting of the lemma, it follows that there is a unique generator~$\L_\As(\X,\pi)$ which is an Euler factor and that it is independent of the character $\psi$. 
We call~$\L_\As(\X,\pi)$ the \emph{Asai~$\L$-factor} of~$\pi$.

For an element~$s(\X)\in\Zl(\X)$ of the form~$1/\P(\X)$ with~$\P(\X)\in\Zl[\X]$ we write~$r_{\ell}(s(\X))=1/r_{\ell}{(\P(\X))}$, and if~$s'(\X)\in\Zl(\X)$ is of the form~$1/\P'(\X)$ with~$\P'(\X)\in\Zl[\X]$ we write~$s(\X)\mid s'(\X)$ if~$\P(\X)\mid \P'(\X)$.

\begin{lemma}\label{Lfactordivision}
Let~$\pi$ be an integral cuspidal~$\Ql$-representations of~$\G$ and~$\overline{\pi}$ its reduction modulo~$\ell$.  
\begin{enumerate}
\item Then~$\L_\As(\X,\pi)$ is the inverse of a polynomial in~$\Zl[\X]$.
\item Moreover,
\[\L_\As(\X,\overline{\pi})\mid r_{\ell}(\L_\As(\X,\pi)).\]
\end{enumerate}
\end{lemma}
\begin{proof}
The first part (i) is in fact true more general for not-necessarily integral representations of Whittaker type, and follows from the asymptotic expansion of Whittaker functions as in \cite[Corollary 3.6]{KM17}.  The second part (ii) follows by imitating the proof of \cite[Theorem 3.13]{KM17}, we recall the argument here:  By definition, we can write the~$\L$-factor~$\L_\As(\X,\overline{\pi})$ as a finite sum of Asai integrals:  For~$i\in\{1,\ldots,r\}$, there are~$\Phi_i\in\Cc_c^\infty(\F_\so^n)$ and~$\W_i\in\Ww(\overline{\pi},\overline{\psi})$ such that
\[\L_\As(\X,\overline{\pi})=\sum_{i=1}^r \I_{\As}(\X,\Phi_i,\W_i).\]
By \cite[Lemma 2.23]{KM17}, there are Whittaker functions~$\W_{i,e}\in \Ww(\pi,\psi)$ which take values in~$\Zl$ such that~$\W_i=r_{\ell}(\W_{i,e})$, and clearly there are Schwartz functions~$\Phi_{i,e}\in\Cc_c^\infty(\F_\so^n)$ which take values in~$\Zl$ such that~$\Phi_i=r_{\ell}(\Phi_{i,e})$.  Moreover,
\[\sum_{i=1}^r\I_{\As}(\X,\Phi_{i,e},\W_{i,e})\in \L_{\As}(\X,\pi)\Ql[\X^{\pm 1}]\cap \Zl((\X))=\L_{\As}(\X,\pi)\Zl[\X^{\pm 1}],\]
hence~$\L_{\As}(\X,\overline{\pi})=\sum_{i=1}^r \I_{\As}(\X,\Phi_i,\W_i)\in r_{\ell}(\L_{\As}(\X,\pi))\Fl[\X^{\pm1}]$.
\end{proof}
As we shall see later, strict divisions do occur.

\subsection{Test vectors}\label{section test vectors}
In \cite{AKMSS}, test vectors for the Asai integral of a distinguished supercuspidal~$\Ql$-representation were given with the Asai integral computed explicitly.  
The pro-order of a compact open subgroup of~$\G_\so$ may be zero in~$\Fl$, and so one cannot normalise a right Haar measure with values in~$\Fl$ arbitrarily. So,
for compatibility with reduction modulo~$\ell$, we need to be more careful with normalisation of measures over~$\Ql$.  
We set~$\K=\GL_n(\oo)$ and $\K_\so=\K^\sigma$,~$\K_\so^1=\I_n+\mathcal{M}_n(\oo_{\so})$, and~$\P$ the~$\sigma$-stable mirabolic subgroup of~$\G$ of all elements with final row~$(\begin{smallmatrix}0&\cdots &0&1\end{smallmatrix})$ and.

\begin{definition}
A triple $(\bJ,\bl,\psi)$ with $(\bJ,\bl)$ a $\sigma$-self-dual $\R$-type, and $\psi$ a distinguished nondegenerate character of $\N$ satisfying conditions (i) and (ii) of \cite[Lemma 6.8]{AKMSS} will be called an \textit{adapted type}. 
\end{definition}

\begin{remark}
\begin{enumerate}
\item In particular if $(\bJ,\bl,\psi)$ is an adapted type, the type $(\bJ,\bl)$ is a $\sigma$-self-dual $\psi$-generic type (in particular $\psi$ is distinguished).
\item By the proof of \cite[Lemma 6.8]{AKMSS}, if $\pi$ is a $\sigma$-self-dual cuspidal $\R$-representation, it contains an adapted type, the point of the remark being that 
the $\N$ of [ibid.] can be chosen to be our $\N$: the group of unipotent upper triangular matrices in $\G$.
\end{enumerate}
\end{remark}

Now let $\pi$ be a $\sigma$-self-dual cuspidal $\R$-representation, and $(\bJ, \bl, \psi)$ be an adapted type of $\pi$. We associate to $(\bJ,\bl,\psi)$ 
the \emph{Paskunas--Stevens Whittaker function} $\W_{\bl}\in \mathcal{W}(\pi,\psi)$ defined in
 \cite[(6.3)]{AKMSS}. Note that,~$\W_{\bl}$ takes values in 
in~$\Zl$ as soon as~$\pi$ is integral (see for example \cite[Lemma 10.2]{KMNagoya}). One of the main results of \cite{AKMSS} is that 
this Whittaker function is a test vector for the Asai $\L$-factor:

\begin{proposition}[{\cite[Theorem 7.14]{AKMSS}}]\label{proptestvectors}
Let~$\pi$ be a distinguished supercuspidal~$\Ql$-representation of~$\G$, and~$\W_{\bl}$ be the explicit Whittaker function defined above. There is a unique nomalisation of the invariant measure on~$\N_\so\backslash \G_\so$ such that
\[\I_\As(\X,\mathbf{1}_{\o_{\F_\so^n}},\W_{\bl})=
(q_\so-1)(q_\so^{n/e_\so(\pi)}-1)\L_{\As}(\X,\pi).\] The volume of $\N_\so\cap \K_\so^1\backslash \K_\so^1$ is of the form $p^l$ for $l\in\mathbb{Z}$ with this normalisation.
\end{proposition}

\begin{proof}
We start with Haar measures~$dg$ and~$dn$ on~$\G_\so$ with values in~$\Ql$ normalised by~$dg(\K_\so^1)=1$ and~$dg(\N_\so\cap \K_\so^1)=1$, which in turns normalises the measure (still denoted~$dg$) on the quotient~$\N_\so\backslash \G_\so$.

With this normalisation, which is the exact parallel of the normalisation used in \cite{KMNagoya} for the analogue Rankin-Selberg computation, first of all we get an extra factor of~$(q_\so-1)$ on the top of the Tate factor defined before~\cite[Lemma 7.11]{AKMSS}. 
Then there is a factor~$dk((\P^\sigma\cap\K^\sigma)\backslash \J^\sigma)$ which appears in~\cite[Lemma 6.11]{AKMSS}, and we have
\[dk((\P^\sigma\cap\J^\sigma)\backslash \J^\sigma)=dk((\P^\sigma\cap(\J^1)^\sigma)\backslash (\J^1)^\sigma)|\J^\sigma/(\P^\sigma\cap\J^\sigma)(\J^1)^\sigma)|.\]
As~$dk((\P^\sigma\cap(\J^1)^\sigma)\backslash (\J^1)^\sigma)$ is a (possibly negative) power of~$p$ we can renormalise our measure to remove it. The image of~$\P\cap \J$ modulo~$\J^1$ is a~$\sigma$-stable mirabolic~$\overline{\P}_m(\kk_\E)$ of~$\J/\J^1$, and we thus have 
\begin{equation*}
|\J^\sigma/(\P^\sigma\cap\J^\sigma)(\J^1)^\sigma)|=|\GL_m(\kk_\E)^\sigma/\overline{\P}_m(\kk_\E)^\sigma|=q_0^{n/e_\so(\pi)}
-1\end{equation*}
thanks to Lemma \ref{comparison}.
\end{proof}

\begin{corollary}\label{cor rel banal and red mod ell of the L factor}
Suppose that~$\pi$ is an unramified twist of a relatively banal distinguished cuspidal~$\Fl$-representation, and let~$\widetilde{\pi}$ be a supercuspidal lift of~$\pi$.  Then
\[\L_\As(\X,\pi)=r_{\ell}(\L_\As(\X,\widetilde{\pi})).\]
\end{corollary}

\begin{proof}
Let $\widetilde{\pi}$ be such a lift, thanks to Theorem \ref{thm relatively banal distinction lift} there is an 
unramified character~$\widetilde{\chi}$ of $\F^\times$ such that $\widetilde{\pi}_0=(\widetilde{\chi}\circ \det)^{-1}\otimes \widetilde{\pi}$ 
is distinguished. Let $(\bJ,\widetilde{\bl},\widetilde{\psi})$ be an adapted type of $\widetilde{\pi}_\so$. 
Proposition \ref{proptestvectors} then implies that 
\[\I_\As(\X,\mathbf{1}_{\o_{\F_\so^n}},\W_{\bl})=
(q_\so-1)(q_\so^{n/e_\so(\pi)}-1)\L_\As(\X,\pi_\so).\] 
Then setting $\pi_\so=r_\ell(\widetilde{\pi}_\so)$,  we deduce that \[\L_\As(\X,\pi_\so)=r_{\ell}(\L_\As(\X,\widetilde{\pi}_\so))\]
in the exact same way that \cite[Corollary 10.1]{KMNagoya} follows from \cite[Proposition 9.3]{KMNagoya}. We obtain the statement of the corollary by twisting $\widetilde{\pi}_\so$ by $\widetilde{\chi}\circ \det$ in this equality as it sends $\X$ to $\chi(\w_\so)\X$ on the left hand side and to $\widetilde{\chi}(\w_\so)\X$ on the right hand side.
\end{proof}

\subsection{Asai~$\L$-factors of cuspidal~representations}\label{section computation Asai}

We first recall the computation of the Asai~$\L$-function of a cuspidal~$\Ql$-representation:

\begin{proposition}[{\cite[Corollary 7.6]{AKMSS} and Remark 7.7}]\label{prop elladicLfactor}
Let~$\pi$ be a cuspidal~$\Ql$-representation.  If no unramified twist of~$\pi$ is distinguished then~$\L_{\As}(\X,\pi)=1$.  If~$\pi$ is distinguished then
\[\L_{\As}(\X,\pi)=\frac{1}{1-\X^{n/e_\so(\pi)}}.\]

\end{proposition}
This gives a complete description in the cuspidal case, as for an unramified character~$\chi:\F^\times\rightarrow \K^\times$ we have
\[\L_{\As}(\X,(\chi\circ\det)\otimes\pi)=\L_{\As}(\chi(\w_\so) \X,\pi).\]

\begin{theorem}\label{thm AsaiLfunctioncomputation}
Let~$\pi$ be a cuspidal~$\Fl$-representation of~$\GL_n(\F)$.  
\begin{enumerate}
\item If~$\pi$ is an unramified twist~$(\chi\circ\det)\otimes\pi_0$ of a relatively banal distinguished representation~$\pi_0$, then 
\[\L_{\As}(\X,\pi)=\frac{1}{1-(\chi(\w_\so)\X)^{n/e_\so(\pi)}}.\]   
\item If~$\pi$ is not an unramified twist of a relatively banal distinguished representation, then\[\L_{\mathrm{As}}(\X,\pi)=1.\]
\end{enumerate}
\end{theorem}

\begin{proof}
If~$\pi$ is an unramified twist of a relatively banal distinguished representation the statement follows for example from 
Corollary \ref{cor rel banal and red mod ell of the L factor} and Proposition \ref{prop elladicLfactor}.  

If~$\pi$ is not an unramified twist of a relatively banal distinguished representation, it has a supercuspidal lift~$\widetilde{\pi}$ which is not an unramified twist of a distinguished representation thanks to Theorem \ref{thm relatively banal distinction lift}. By Proposition \ref{prop elladicLfactor} we have~$\L_\As(\X,\widetilde{\pi})=1$, hence~$\L_\As(\X,\pi)=1$ as~$\L_\As(\X,\pi)\mid r_{\ell}(\L_\As(\X,\widetilde{\pi}))$ by Lemma \ref{Lfactordivision}.
\end{proof}

\begin{remark}
Note that when~$\ell=2$, then we are in case (ii) of Theorem \ref{thm AsaiLfunctioncomputation} and 
$\L_{\As}(\X,\pi)=1$. This can also be seen directly from the asymptotics of Whittaker functions. Without entering the details as we don't need this, the asymptotic expansion of Whittaker functions on the diagonal torus allow one to express the Asai integrals in terms of Tate integrals for~$\F_\so^\times$, and these Tate integrals are all~$1$ because~$q=1[2]$ as shown in \cite{MinguezZeta}.
\end{remark}

\section{Distinction and poles of the Asai~$\L$-factor}\label{section pole Asai}

\subsection{Characterisation of the poles of the Asai~$\L$-factor}\label{section characterisation of pole of Asai}

We are now in position to prove the main results of this paper:

\begin{theorem}\label{thm Pole of Asai}
Let~$\pi$ be a cuspidal~$\Fl$-representation of~$\G$. Then~$\L_\As(\X,\pi)$ has a pole at~$\X=1$ if and only if~$\pi$ is relatively banal distinguished. In this case, the pole is of order~$\ell^r$ where~$n/e_\so(\pi)=a\ell^r$ with~$a$ prime to~$\ell$.
\end{theorem}
\begin{proof}
 If~$\L_\As(\X,\pi)$ has a pole at~$\X=1$, in particular~$\L_\As(\X,\pi)$ is not equal to~$1$, hence the representation~$\pi$ is an unramified twist of a relatively banal distinguished cuspidal~$\Fl$-representation~$\pi_0$, say~$\pi=(\chi\circ \det)\otimes \pi_0$ with~$\chi$ an unramified characer of $\F^\times$. Denote by $\chi_\so$ the restriction of $\chi$ to $\F_\so^\times$,  then by Theorem \ref{thm AsaiLfunctioncomputation},
\[\L_{\As}(\X,\pi)=\frac{1}{1-(\chi_\so(\w_\so)\X)^{n/e_\so(\pi)}}=\frac{1}{(1-(\chi_\so(\w_\so)\X)^{a})^{\ell^r}}\] which has a pole at~$\X=1$ if and only if~$\chi_\so(\w_\so)^{n/e_\so(\pi)}=1$. By Theorem \ref{thm Xso} this implies that 
$\chi_\so\circ \det$ belongs to $\X_\so(\pi)$, i.e. that~$\pi=(\chi\circ \det)\otimes \pi_0$ is distinguished. The converse is just Theorem \ref{thm AsaiLfunctioncomputation}.
\end{proof}

\begin{remark}\label{remark Kable}
Note that our proof of Theorem \ref{thm Pole of Asai} is very different from the proof over the field of complex numbers. In the proof above, the direction~$\pi$ relatively banal distinguished implies~$\L_\As(\X,\pi)$ having a pole at~$1$ is an immediate consequence of Theorem \ref{thm AsaiLfunctioncomputation}, and works for complex representations as well (in which case we consider all cuspidal distinguished~$\mathbb{C}$-representations as ``relatively banal") thanks to \cite[Corollary 7.6]{AKMSS}. Saying this is not enough to claim a proof in the case of complex cuspidal representations different from the original one given in \cite[Theorem 1.4]{AKT}, as the way the equality of \cite[Corollary 7.6]{AKMSS} is obtained is a consequence of \cite[Proposition 6.3]{MatringeManuscripta}, which itself follows either from \cite[Theorem 3.1]{MatringeManuscripta} or from \cite[Theorem 1.4]{AKT} and \cite[Theorem 4]{Kable}, together with the fact that the poles of the Asai~$\L$-factor are simple in the cuspidal case. However,  the first equality in \cite[Theorem 7.14]{AKMSS} is independant of the results cited above, and it in particular implies that if~$\pi$ is a cuspidal distinguished~$\mathbb{C}$-representation, its Asai~$\L$-factor has a pole at~$\X=1$.
The proof of the other implication that we give also works in the complex case, and is again different from the original proof given in \cite[Theorem 4]{Kable}. Kable shows that if~$\L_\As(\X,\pi)$ has a pole at~$\X=1$, the rational function~$(1-\X)\I_{\As}(X,\W,\Phi)$ is regular at 
$\X=1$ and that up to a nonzero constant independant of~$\W$ and~$\Phi$ that its value at~$\X=1$ is given by 
\[\Phi(0)\int_{\Z^\sigma \N^\sigma\backslash \G^\sigma} \W(h)dh.\] As by assumption the Asai~$\L$-factor has a pole at~$\X=1$ the~$\G^\sigma$-invariant linear form  \[\mathcal{L}_\pi: \W\mapsto \int_{\Z^\sigma \N^\sigma\backslash \G^\sigma} \W(h)dh\] is nonzero.
Note that to adapt this proof to the modular setting with~$\R=\Fl$ we would need to take~$1-\X^{n/e_\so(\pi)}$ where he takes~$1-\X$ (this does not matter over~$\mathbb{C}$ as both polynomials have a simple zero at~$\X=1$) to get the correct order of the pole, though from Kable's proof one sees that the natural choice is in fact~$1-\X^n$. But we claim that this can't be done in general, as we shall now see that the local period~$\mathcal{L}_\pi$, though well defined for cuspidal~$\Fl$-representations might vanish even for 
relatively banal distinguished cuspidal~$\Fl$-representations.
\end{remark}

\subsection{The~$\G_{\so}$-period of cuspidal distinguished representations}\label{section the period of cuspidal distinguished reps}\label{section the period}
Let~$\pi$ be a cuspidal distinguished~$\R$-representation, and we still denote by $\psi$ a distinguished nondegenerate character of $\N$.  
There are two natural~$\G^\sigma$-invariant linear forms on 
$\mathcal{W}(\pi,\psi)$. The first is \[\mathcal{P}_\pi:\W\mapsto \int_{\N^\sigma\backslash \P^\sigma} \W(p)dp\] which is well-defined and nonzero 
thanks to \cite[Chapter III, Theorem 1.1]{Vig96}. Though it does not look~$\G^\sigma$-invariant it is by \cite[Proposition B.23]{AKMSS}.  
Let $(\bJ,\bl, \psi)$ be an adapted type in~$\pi$. A natural test vector for this linear form is the Paskunas-Stevens 
Whittaker function~$\W_{\bl}$: we have~$\mathcal{P}_\pi(\W_{\bl})\neq 0$ according to the proof of \cite[Proposition 6.5]{AKMSS}.

The second is \[\mathcal{L}_\pi: \W\mapsto \int_{\Z^\sigma \N^\sigma\backslash \G^\sigma} \W(h)dh.\] It is~$\G^\sigma$-invariant by definition, and well defined, as all~$\W\in \mathcal{W}(\pi,\psi)$ have compact support on~$\N\backslash \P$, they have compact support of~$\Z\N\backslash \G$ thanks to the Iwasawa decomposition~$\G=\P\Z\K$. By cuspidal multiplicity one for the pair $(\P, \P^\sigma)$ (\cite[Proposition B.23]{AKMSS}),~$\mathcal{L}_\pi$ is a multiple of~$\mathcal{P}_{\pi}$ and the proportionality constant between them turns out to be 
a very interesting quantity; this scalar is related to the formal degrees of complex discrete series representations of unitary groups (see \cite{Anand-Matringe18} and Remark \ref{remarklinearforms}).  When~$\R=\Ql$ the linear form~$\mathcal{L}_\pi$ is nonzero (Remark \ref{remark Kable}); here, we solve the problem of understanding when~$\mathcal{L}_\pi$ is nonzero when~$\R=\Fl$:

\begin{theorem}\label{thm nonvanishingperiod}
Let~$\pi$ be a cuspidal distinguished~$\Fl$-representation of~$\G$, then~$\mathcal{L}_\pi$ is nonzero if and only if:
\begin{enumerate}
\item~$\pi$ is relatively banal.
\item~$\ell$ does not divide~$e_\so(\pi)$.
\end{enumerate}
\end{theorem}
\begin{proof}
Thanks to the Iwasawa decomposition~$\G=\P\Z\K$, we have the equality:
\[\int_{\Z^\sigma\N^\sigma\backslash \G^\sigma} \W(h)dh=\int_{\K^\sigma\cap \P^\sigma\backslash\K^\sigma}\int_{\N^\sigma\backslash \P^\sigma} \W(pk)|\det(pk)|_\so^{-1}dp dk.\] We introduce the power series 
\[\I_{\As,(0)}(\X,\W)=\sum_{l\in \mathbb{Z}} \left(\int_{\N^\sigma\backslash {\P^\sigma}^{(l)}} \W(p)|\det(p)|_\so^{-1} dp\right)\X^l\] where 
${\P^\sigma}^{(l)}=\{p\in \P^\sigma,\ \val_{\F_\so}(\det(p))=l\}$ which is in fact a Laurent polynomial as~$\pi$ is cuspidal, so that 
\[\mathcal{P}_\pi(\W)=\I_{\As,(0)}(1,\W).\]

Now suppose that~$\pi$ is not relatively banal, then~$\pi$ is~$|\det(~\cdot~)|_\so$-distinguished, and appealing to \cite[Proposition B.23]{AKMSS}, it means that the linear form \[\mathcal{P}_{\pi,|\det(~\cdot~)|_\so}:\W\mapsto \int_{\N^\sigma\backslash \P^\sigma} \W(p)|\det(p)|_\so^{-1}dp\] is 
$|\det(~\cdot~)|_\so$-equivariant under the action of~$\G_\so$.  So in particular, up to possible renormalisation of the invariant measure,
\[\int_{\Z^\sigma\N^\sigma\backslash \G^\sigma} \W(h)dh=
\mathrm{vol}(\K^\sigma\cap \P^\sigma\backslash\K^\sigma)\mathcal{P}_{\pi,|\det(~\cdot~)|_\so}(\W)=(q_\so^n-1)\mathcal{P}_{\pi,|\det(~\cdot~)|_\so}(\W)=0\] as 
$q_\so^{n/e_\so(\pi)}=1$.

So it remains to understand what happens when~$\pi$ is relatively banal. As we said by multiplicity one we know that 
$\mathcal{L}_\pi=\lambda \mathcal{P}_\pi$ and we noticed that~$\mathcal{P}_\pi(\W_{\bl})\neq 0$. Hence~$\mathcal{L}_\pi=0$ if and only if 
$\mathcal{L}_\pi(\W_{\bl})=0$. However, following the proof of \cite[Theorem 9.1]{KMNagoya} at the end of p.19 of [ibid.] or the proof of \cite[Theorem 7.14]{AKMSS}, one gets up to a possible renormalisation of invariant measures: 
\[\mathcal{L}_{\pi,\X}(\W_{\bl}):=\int_{\K^\sigma\cap \P^\sigma\backslash\K^\sigma}\I_{\As,(0)}(\X,\rho(k)\W_{\bl})dk=
(q_\so-1)(q_\so^{n/e_\so(\pi)}-1)\frac{1-\X^n}{1-\X^{n/e_\so(\pi)}}.\]
Now the value at~$\X=1$ of~$\mathcal{L}_{\pi,\X}$ is~$\mathcal{L}_\pi$, so~$\mathcal{L}_\pi$ vanishes if and only if 
$\frac{1-\X^n}{1-\X^{n/e_\so(\pi)}}$ vanishes at~$\X=1$. However the order of the zero of~$1-\X^n$ is the 
$\ell^{\mathrm{val}_\ell(n)}$ whereas that of of the zero of~$1-\X^{n/e_\so(\pi)}$ is~$\ell^{\mathrm{val}_\ell(n/e_\so(\pi))}$. 
This means that if~$\pi$ is relatively banal,~$\mathcal{L}_\pi$ is nonzero if and only if 
$\mathrm{val}_\ell(n)=\mathrm{val}_\ell(n/e_\so(\pi))$, i.e. if and only if~$\ell$ does not divide~$e_\so(\pi)$.
\end{proof}

\begin{remark}\label{remarklinearforms}
Here we explain how this vanishing result modulo~$\ell$ is related to the vanishing of 
the $\ell$-adic proportionality constant between $\mathcal{L}$ and $\mathcal{P}$.  For~an algebraically closed field~$\mathbf{C}$, write~$\Cusp_{\mathbf{C},\text{dist}}(\G)$ for the set of isomorphism classes of distinguished cuspidal~$\mathbf{C}$-representations.  Fix an isomorphism~$\mathbb{C}\simeq \Ql$, this induces a bijection~$\Cusp_{\mathbb{C},\text{dist}}(\G)\rightarrow \Cusp_{\Ql,\text{dist}}(\G)$, independent of the choice of isomorphism.

Let~$\pi\in \Cusp_{\Ql,\text{dist}}(\G)$ and~$\psi:\N\rightarrow\Zl^\times$ be an~$\N^\sigma$-distinguished nondegenerate character of~$\N$.  Then by \cite[Corollary 1.2]{AKT}, there exists~$\mu\in\Ql$ such that
\begin{equation}\label{multoneequation}
\mathcal{L}_\pi=\mu\mathcal{P}_\pi,
\end{equation}

Let~$c_\pi$ denote the central character of~$\pi$ and~$\Res_\P$ denote the restriction of Whittaker functions to~$\P$.  Then \[\mathcal{W}(\pi,\psi)\subset \ind_{\Z\N}^{\G}(c_{\pi}\otimes \psi)\text{ and }\Res_{\P}(\mathcal{W}(\pi,\psi))\subset \ind_{\N}^{\P}(\psi),\] the first fact being a consequence of the second, which has been known since \cite{BZ76}.  Now let $\mathcal{W}(\pi,\psi)_e$ denote the $\Zl$-submodule of $\mathcal{W}(\pi,\psi)$ consisting of Whittaker functions with values in $\Zl$, it follows from \cite[Theorem 2]{VigIntegral} and \cite[Theorem 2]{Vig-Kirillov} that $\mathcal{W}(\pi,\psi)_e$ is a lattice in $\mathcal{W}(\pi,\psi)$ and $\Res_{\P}(\mathcal{W}(\pi,\psi)_e)= 
 \ind_{\N}^{\P}(\psi,\Zl)$ is a lattice in $\Res_{\P}(\mathcal{W}(\pi,\psi))$,
 reducing to $\mathcal{W}(\pi,\psi)$ and $\Res_{\P}(\mathcal{W}(\pi,\psi))$ respectively.
 
 Finally from \cite[Section 2.2]{KM17}, there are appropriate $\ell$-adic and $\ell$-modular invariant measures on $\Z^\sigma\N^\sigma\backslash \G^\sigma$ and $\N^\sigma\backslash \P^\sigma$ such that 
\begin{align}
\label{al1} r_\ell(\mathcal{L}_{\pi}(\W_e))&=\mathcal{L}_{r_{\ell}(\pi)}(r_\ell(\W_e)),\\
\label{al2}r_\ell(\mathcal{P}_{\pi}(\W_e))&=\mathcal{P}_{r_{\ell}(\pi)}(r_\ell(\W_e))
\end{align}
 for all $\W_e\in \W(\pi,\psi)_e$ and 
$\mathcal{P}_{\pi}(\Res_{\P}(\mathcal{W}(\pi,\psi)_e))=\Zl$. 

Evaluating Equation \ref{multoneequation} on an element~$\W_e\in\mathcal{W}(\pi,\psi)_e$ such that $\mathcal{P}_{\pi}(\W_e)=1$ we deduce that~$\mu\in \Zl$.
Now Theorem \ref{thm nonvanishingperiod} and Equations~\ref{al1} and \ref{al2} imply that $\ell$ divides $\mu$ if and only if either~$r_{\ell}(\pi)$ is not relatively banal or~$\ell\mid e_\so(\pi)$.  Running over all~$\ell\neq p$, we recover the radical of the~$p$-regular part of~$\mu$, (explicitly using the type theoretic definition of relatively banal, Definition \ref{definition relbanaldef}). 

As mentioned already, this scalar~$\mu$ is a very interesting and subtle quantity: by \cite[Theorem 7.1]{Anand-Matringe18}, we have
\begin{equation*}
\mathcal{L}_\pi=\lambda \frac{d(\rho)}{d(\pi)}\mathcal{P}_\pi,
\end{equation*}
where~$\lambda$ is a constant independent of~$\pi$,~$\rho$ is the cuspidal~$\Ql$-representation of the quasi-split unitary group in~$n$-variables defined over~$\F_\so$ which base changes to~$\pi$ (stably or unstably depending on the parity of~$n$), and~$d(\rho)$ and~$d(\pi)$ denote the formal degrees of~$\rho$ and~$\pi$ respectively under the normalisation of invariant measures of \cite{HII08}.  One could check that the formal degrees are rational for our well chosen measures (and $\lambda$ as well), and preserved under the bijection~$\Cusp_{\mathbb{C},\text{dist}}(\G)\rightarrow \Cusp_{\Ql,\text{dist}}(\G)$.

While we have explained how~Theorem \ref{thm nonvanishingperiod} tells us exactly when~$\mu$ vanishes modulo $\ell$, we could also go in the other direction:  By \cite{HII08} and \cite{BP18}, the constant $\mu$ could be computed explicitly and its explicit description would give a different proof of Theorem \ref{thm nonvanishingperiod}. It should be clear to the reader that the amount of work required for such a proof is much more considerable than that of the proof given above.
\end{remark}

\subsection{Comparison of banal and relatively banal}\label{section comparison banal rel banal}
Finally, we compare our notion of relatively banal distinguished with the notion of banal representation introduced in \cite{MSDuke} for cuspidal representations.

By \cite[Remarque 8.15]{MSDuke}, a cuspidal~$\Fl$-representation $\pi$ of~$\G_\so$ is banal if and only if \[|\det(\ )|_\so\otimes \pi\not\simeq \pi.\] However the map~\[b:\pi\mapsto \pi\otimes \pi^{\vee}\] is a bijection between the set of (isomorphism classes of) irreducible representations of~$\G_\so$ and the set of~$\Delta(\G_\so)$-distinguished irreducible representations of~$\G'=\G_\so\times \G_\so$, where~$\Delta$ is the diagonal embedding of $\G_\so$ into $\G'$. In particular, 
$\pi$ (seen as the distinguished representation~$b(\pi)$ of~$\G'$) is banal if and only if 
$|\det(\ )|_\so\otimes  b(\pi)=(|\det(\ )|_\so\otimes \pi)\otimes \pi^\vee$ is not distinguished. Note that~$|\ |_\so$ plays the same role for the split quadratic algebra~$(\F_\so\times \F_\so)/\F_\so$ that 
$|\ |_{\so}$ plays for~$\F/\F_\so$, i.e. it is a square root of the absolute value on the bigger algebra. So this proves the exact analogy of banal cuspidal representations of $\G_\so$ and relatively banal distinguished cuspidal representations of $\G$ according to Corollary \ref{cor relatively banal simple definition}.  

The analogy can also be seen at the~$\L$-factor level: it follows from \cite[Theorem 4.9]{KMNagoya} that if 
$\pi\otimes \pi'$ is a cuspidal representation of~$\G'$, then the Rankin--Selberg~$\L$-factor~$\L(\X,\pi,\pi')$ (which can be thought of as the Asai~$\L$-factor of~$\pi\otimes \pi'$) has a pole at~$\X=1$ if and only if~$\pi\otimes \pi'$ is~$\Delta(\G_\so)$-distinguished and~$\pi$ is banal, which is the exact analogue of Theorem \ref{thm Pole of Asai} replacing banal with relatively banal.

Finally, in terms of the type theory definition, a cuspidal representation $\pi$ of $\G_\so$ is banal if and only if $q_\so^{n/e(\pi)}\not \simeq 1[\ell]$, but tracking down how $e(\pi)$ is defined with respect to $\pi$ in terms of type theory (more precisely lattice periods) shows that it plays 
the same role for the $\Delta(\G_\so)$-distinguished representation $b(\pi)$ of $\G'$ that $e_\so(\tau)$ plays for a distinguished cuspidal representation $\tau$ of $\G$. 

\bibliographystyle{plain}
\bibliography{Distinguished}
\end{document}